\documentclass[a4paper,12pt%
]{amsart}
\pdfoutput=1
\usepackage{amsmath, amssymb, amsthm, amscd}
\usepackage{graphicx}
\usepackage{enumerate}
\usepackage[all,cmtip]{xy}

\newcounter{theorem}
\newtheorem{thm}[theorem]{Theorem}
\newtheorem{lemma}[theorem]{Lemma}
\newtheorem{prop}[theorem]{Proposition}
\newtheorem{cor}[theorem]{Corollary}
\newtheorem{defn}[theorem]{Definition}

\theoremstyle{remark}
\newtheorem*{remark*}{Remark}
\newtheorem{remark}[theorem]{Remark}

\newtheorem{question}[theorem]{Question}
\theoremstyle{theorem}

\numberwithin{equation}{section}
\numberwithin{theorem}{section}

\DeclareMathOperator\spa{span}
\DeclareMathOperator\Mon{Mon}
\DeclareMathOperator\rank{rank}

\DeclareMathOperator\Tr{Tr}

\newcommand{\ndiv}{\hspace{-4pt}\not|\hspace{2pt}}

\newcommand{\defemph}{\emph}

\newcommand{\R}{\mathbb{R}}
\newcommand{\Q}{\mathbb{Q}}
\newcommand{\Z}{\mathbb{Z}}
\newcommand{\C}{\mathbb{C}}
\newcommand{\N}{\mathbb{N}}

\renewcommand{\setminus}{\backslash}

\title[Simple stationary limits]{Simple dimension groups that are isomorphic to stationary inductive limits}

\author{Gregory R. Maloney}
\address{Newcastle University}
\date{\today}

\begin{document}

\begin{abstract}
A dimension group is an ordered abelian group that is an inductive limit of a sequence of simplicial groups, and a stationary dimension group is such an inductive limit in which the homomorphism is the same at every stage.  
If a simple dimension group is stationary then up to scalar multiplication it admits a unique trace (positive real-valued homomorphism), but the short exact sequence associated to this trace need not split.  
In an earlier paper, Handelman described these ordered groups concretely in the case when the trace has trivial kernel---i.e., the group is totally ordered---and in the case when the group is free.  
The main result here is a concrete description of how a stationary simple dimension group is built from the kernel and image of its trace.  
Specifically, every stationary simple dimension group contains the direct sum of the kernel of its trace with a copy of the image, and is generated by that direct sum and finitely many extra elements.  
Moreover, any ordered abelian group of this description is stationary.  

The following interesting fact is proved along the way to the main result: given any positive integer $m$ and any square integer matrix $B$, there are two distinct integer powers of $B$, the difference of which has all entries divisible by $m$.  
\end{abstract}

\maketitle

\section{Introduction}\label{SEC:intro}

\begin{defn}\label{DEF:dimension-group}
An ordered Abelian group is called a \defemph{dimension group} if it is isomorphic to the inductive limit of a sequence of simplicial groups (direct sums of finitely many copies of $\Z$) in the category of ordered Abelian groups.  
\end{defn}

The order structure of an ordered group $G$ is determined by its positive cone $G^+ := \{ g\in G : g \geq 0\}$.  
Let us assume that all ordered groups are \defemph{directed}, meaning that $G = G^+ - G^+$.  

\begin{defn}\label{DEF:stationary}
A \defemph{stationary} inductive sequence is a sequence of the form 
\centerline{
\xymatrix{
\Z^k \ar@{>}^{A}[r] & \Z^k \ar@{>}^{A}[r] & \Z^k \ar@{>}^{A}[r] & \Z^k \ar@{>}^{A}[r] & \Z^k \ar@{>}^{A}[r] & \cdots \\
}
}
in which the homomorphism $A:\Z^k\to \Z^k$ is the same at each stage, and an ordered abelian group is called \defemph{stationary} if it is isomorphic to the inductive limit of a stationary sequence with a positive homomorphism $A$.  
\end{defn}

Every stationary group is a dimension group; the question considered here is that of describing the simple dimension groups that are stationary.  
This is essentially a question of finding the range of the invariant for certain classes of topological and dynamical objects, most notably the shifts of finite type (also called topological Markov chains), which are fundamental objects of study in the theory of dynamical systems.  
Every subshift of finite type has an associated stationary dimension group, which is an invariant of the subshift \cite{K:dimension-group}; simplicity of this dimension group is equivalent to the subshift being mixing.  
Stationary dimension groups also arise as cohomological invariants of substitution tiling spaces; more will be said about this in Section \ref{SEC:tiling-cohomology}.  

The question of how to describe simple stationary dimension groups has already been answered in \cite{H:irrational} in the free case and in the non-free totally ordered case.  
But in the case of a dimension group that is neither free nor totally ordered, the short exact sequence associated to the \defemph{trace} (order-preserving real-valued functional, normalized on any fixed positive element) need not split.  
The kernel of the trace is a finite-rank torsion-free abelian group---hence a subgroup of $\Q^r$---and the image of the trace is a simple totally ordered abelian group---hence a subgroup of $\R$---but the question remains of how these groups are combined, assuming they are stationary in the unordered and ordered sense respectively, to produce a stationary ordered group.  
This is the question that is answered in the following theorem, which is the main result.  

\newtheorem*{THM:description}{Theorem \ref{THM:description}}
\begin{THM:description}
Let $G$ be a non-cyclic simple dimension group.  
Then $G$ is stationary if and only if it is order isomorphic to a subgroup of $\R \oplus \Q^{r}$ ordered by the first coordinate and generated by the following:  
\begin{enumerate}
\item  a non-cyclic stationary order subgroup $H\subset \R\oplus 0^r$;  
\item  a rank-$r$ subgroup $K \subset \{ (0,q_1,\ldots,q_r) \in\R\oplus \Q^{r} \}$ that is stationary in the category of unordered torsion-free abelian groups; and
\item  a finite set $\{ z_1,\ldots, z_s\}\subset (H + K)\otimes \Q\subset \R \oplus \Q^r$.  
\end{enumerate}
Moreover, the number $s$ of extra generators can be taken to be less than or equal to the minimum of the ranks of $H$ and $K$.  
\end{THM:description}

The difficult part of this theorem is proving that any such group can be realized as a stationary limit with a positive integer matrix.  
Section \ref{SEC:example} contains a worked example in which a matrix is found that realizes the stationary property for a particular dimension group $G$.  

\section{Notation and definitions}\label{SEC:notation}

Let us use the term \defemph{rank} to refer to the torsion-free rank of a torsion-free abelian group $G$, that is, the maximum size of a $\Z$-independent subset of $G$ (alternatively, the dimension of the vector space $G\otimes \Q$).  
Every finite-rank torsion-free abelian group is isomorphic to a subgroup of $\Q^r$.  

Certain order subgroups of $\R^r$ (in fact, $\R \oplus \Q^{r-1}$) are of particular interest in this work.  
Let us say that the ordered group $\R^r$ with positive cone $\{ (x_1,\ldots,x_r) : x_1>0\}\cup \{ 0\}$ is \defemph{ordered by the first coordinate}, and likewise for any order-subgroup $G\subset\R^r$ with positive cone $G^+ = G\cap \{ (x_1,\ldots,x_r) : x_1>0\}\cup \{ 0\}$.  
Not all such order subgroups $G$ are dimension groups; using the famous result \cite[Theorem 2.2]{EHS} that the class of dimension groups coincides exactly with the class of countable torsion-free ordered abelian groups with the Riesz interpolation property, one can check that such a $G$ is a dimension group if and only if it is countable and either it is cyclic with trivial projections on all but the first coordinate, or its projection on the first coordinate is a dense subgroup of $\R$.  

A \defemph{trace} on an ordered abelian group $G$ is a positive group homomorphism $\tau : G \to \R$.  
Up to positive scalar multiples, the only trace on $G\subset \R^r$ ordered by the first coordinate is projection on the first coordinate.  

A dimension group $G$ is \defemph{simple} if, for all $g,h\in G^+$, there exists $n\in\N$ such that $0\leq h \leq ng$.  
It is easy to verify that, if $\tau$ is a trace on a simple dimension group $G$, then $\tau$ must take strictly positive values on $G^+$.  

Every inductive sequence of (ordered or unordered) torsion-free abelian groups has a limit, which has the following standard construction (see also \cite[Exercise 7.6.8]{DF:abstract-algebra}).  
Let 

\centerline{
\xymatrix{
G_1 \ar@{>}^{A_1}[r] & G_2 \ar@{>}^{A_2}[r] & G_3 \ar@{>}^{A_3}[r] & G_4 \ar@{>}^{A_4}[r] & \cdots \\
}
}
be an inductive sequence of torsion-free abelian groups.  
Then the inductive limit of this sequence is isomorphic as a set to the quotient $\{ (g,i) : i\in \N,g\in G_i\}/{\sim}$, where $(g_1,i_1)\sim (g_2,i_2)$ if there exists $j>i_1,i_2$ such that $A_{j-1}\cdots A_{i_1+1}A_{i_1} g_1 = A_{j-1}\cdots A_{i_2+1}A_{i_2}g_2$.  
Let us denote the equivalence class of $(g,i)$ under $\sim$ by $[g,i]$.  
Then the group operation is defined on elements $[g,i], [h,j]$ with $i\geq j$ by 
\begin{align*}
[g,i] + [h,j] & = [g + A_{i-1}\cdots A_jh,i].
\end{align*}
If the groups $G_i$ are all ordered groups and the homomorphisms $A_i$ are all positive (meaning that $A_i(G_i^+)\subset G_{i+1}^+$) then the inductive limit is also an ordered group with positive cone equal to $\{ [g,i] : A_{j-1}\cdots A_i g\in G_j^+$ for some $j\geq i\}$.  

Dimension groups are limits of inductive sequences in which the groups $G_i$ are simplicial groups, i.e., $\Z^k$, ordered by the positive cone $(\Z^k)^+$ consisting of all columns with non-negative entries.  
A homomorphism $A$ between two such groups $\Z^k$ and $\Z^l$ can be represented as an $l\times k$ integer matrix, the columns of which express the image under $A$ of the standard basis elements of $\Z^k$ as combinations of the standard basis elements of $\Z^l$.  
Let us also use the symbol $A$ to denote this matrix.  
If $A$ is a positive group homomorphism, then $A$ is a matrix of non-negative integers.  

So in a sense there is already an answer to the question of what a simple stationary dimension group looks like: it looks like a set of equivalence classes of pairs of indices and groups elements.  
But this is not a useful description; it would be much better to be able to describe such a group concretely, that is, as a subgroup of a real vector space with an appropriate order structure---in this case, ordering by the first coordinate.  
That is what is done here in Theorem \ref{THM:description}.  

Suppose that an ordered group $G$ is isomorphic to the inductive limit of a stationary system:

\centerline{
\xymatrix{
\Z^k \ar@{>}^{A}[r] & \Z^k \ar@{>}^{A}[r] & \Z^k \ar@{>}^{A}[r] & \Z^k \ar@{>}^{A}[r] & \Z^k \ar@{>}^{A}[r] & \cdots. \\
}
}
Then the statement that $G$ is simple is equivalent to the statement that $A$ is \defemph{primitive}, that is, there exists some positive power of $A$, all entries of which are strictly positive.  
The Perron--Frobenius theory then implies that $A$ has a unique eigenvalue $\lambda$ of multiplicity one with maximal modulus and a left $\lambda$-eigenvector $w$ with strictly positive real entries.  
The vector $w$ is a row, so acts on $\Z^k$ by multiplication; this yields a compatible system of positive homomorphisms to $\R$:

\centerline{
\xymatrix{
\Z^k \ar@{>}^{A}[r] \ar@{>}_{w}[dr] & \Z^k \ar@{>}^{A}[r] \ar@{>}^(0.3){w\lambda^{-1}}[d] & \Z^k \ar@{>}^{A}[r] \ar@{>}^{w\lambda^{-2}}[dl] & \cdots \\
& \R & & 
}
}
By the universal property of the inductive limit, this system induces a trace $\tau:G\to \R$.  
It is not difficult to verify that all traces on $G$ are obtained in this way by taking positive multiples of $w$.

\section{An intrinsic characterization of stationarity in the simple case}\label{SEC:intrinsic}

The main result of this section is a proof that stationarity of a simple dimension group $G$ is equivalent to the conditions that $G$ have a unique normalized trace and that there exist a finitely-generated sub-monoid of $G^+$, the union of the images of which under a particular automorphism exhausts all of $G^+$.  
Let us summarize the second of these conditions in a formal definition.  
For this let us introduce the following notation: given a subset $S$ of a group $G$, let $\Mon (S)$ denote the monoid generated by $S$.  

\begin{defn}\label{DEF:monoid-condition}
Let $G$ be a dimension group with positive cone $G^+$, let $\alpha:G\to G$ be an order automorphism, and let $S\subset G^+$ be a subset.  
Let us say that the pair $(\alpha,S)$ satisfy the \defemph{increasing monoid condition} if $G^+$ is the increasing union of the monoids $\Mon (\alpha^n(S))$.  
\end{defn}

Proposition \ref{PROP:characterization}, below, characterizes stationary simple dimension groups using Definition \ref{DEF:monoid-condition}.  

\begin{prop}\label{PROP:characterization}
Let $G$ be a simple dimension group.  
Then $G$ is stationary if and only if it has a unique trace (up to multiplication by a positive scalar) and an order automorphism $\alpha$ and finite subset $F\subset G^+$ that satisfy the increasing monoid condition.  
\end{prop}

Both for the ``only if'' and the ``if'' parts of the proof use the following lemma, which is proved in \cite{H:irrational}, and one direction of which is the well-known Perron--Frobenius theorem.  

\begin{lemma}\cite[Lemma 2.1]{H:irrational}\label{LEM:eventually-positive}
Let $A$ be a square matrix with real entries.  
Then $A$ is primitive if and only if:
\begin{enumerate}
\item  $A$ has a real eigenvalue $\lambda$ of multiplicity one, such that for all other eigenvalues $\mu$ of $A$ in $\C$, $\lambda > |\mu|$ (such an eigenvalue is called a \defemph{weak Perron--Frobenius eigenvalue}); and 
\item  the left and right eigenvectors corresponding to $\lambda$ can be chosen with strictly positive entries.  
\end{enumerate}
\end{lemma}

Let us first give a proof of the necessity of the two conditions in Proposition \ref{PROP:characterization} before stating a lemma that will be used in the proof of their sufficiency.  

\begin{proof}[Proof of Proposition \ref{PROP:characterization}, ``only if'']
Suppose $G$ is isomorphic to a stationary inductive limit with positive homomorphism $A:\Z^k\to \Z^k$ at every stage.  
Then the identity homomorphism $I:\Z^k\to \Z^k$ produces a family of positive homomorphisms from stage $n$ to stage $n+1$ (and hence from stage $n$ to the limit $G$) that make the following diagram commute.  

\centerline{
\xymatrix{
\Z^k \ar@{>}^{A}[r] \ar@{>}^{I}[dr] & \Z^k \ar@{>}^{A}[r] \ar@{>}^{I}[dr] & \Z^k \ar@{>}^{A}[r] \ar@{>}^{I}[dr] & \Z^k \ar@{>}^{A}[r] \ar@{>}^{I}[dr] & %
\cdots \ar@{>}[r] & G \ar@{>}^{\alpha}[d] \\
\Z^k \ar@{>}^{A}[r] & \Z^k \ar@{>}^{A}[r] & \Z^k \ar@{>}^{A}[r] & \Z^k \ar@{>}^{A}[r] & \cdots \ar@{>}[r] & G 
}
}

The universal property of the inductive limit then yields the positive homomorphism $\alpha: G\to G$; $\alpha$ is easily seen to be an order automorphism, and the standard basis elements $\{[e_i,1]\}_{i=1}^k$ of $\Z^k$ from stage $1$ form a finite set of positive elements that satisfies the increasing monoid condition with $\alpha$.  

The homomorphism $A$ is given by left multiplication by a non-negative integer matrix; let us also denote this matrix by $A$.  
If $G$ is simple then some positive integer power $n$ of $A$ has strictly positive entries (so that $[e_i,1] \geq [e_j,n]$ for all $i,j\leq k$).  
Then by Lemma \ref{LEM:eventually-positive} $A$ has a weak Perron--Frobenius eigenvalue $\lambda$, and furthermore the left and right eigenvectors corresponding to this eigenvalue can be chosen with strictly positive entries.  

Let $w$ denote a positive left eigenvector, considered as a row.  
Then, as described in Section \ref{SEC:notation}, the homomorphism $[g,n] \mapsto \lambda^{1-n}wg$ is a trace that can be seen to be unique up to multiplication of $w$ by a positive scalar.  
\end{proof}

To prove the ``if'' part of Proposition \ref{PROP:characterization} involves finding an explicit order endomorphism of $\Z^k$ that realizes the stationary property.  
The following lemma, from \cite{H:almost-ultrasimplicial}, will be useful for this purpose.  

\begin{lemma}\cite[Lemma 1.1]{H:almost-ultrasimplicial}\label{LEM:increasing-semigroups}
Suppose that $G$ is an ordered abelian group with an increasing set of subsemigroups, $S_1\subset S_2\subset \cdots $ such that $G^+ = \bigcup S_n$, with each $S_n$ generated by $\{a_i^{(n)}\}_{i=1}^{k_n}$.  
Suppose that $A_n$ is a transition matrix associated to this choice of generators for $S_n\subset S_{n+1}$; i.e., $a_i^{(n)} = \sum_{j=1}^{k_{n+1}}(A_n)_{ji}a_j^{(n+1)}$.  
(Note the reversed indices $i$ and $j$.)  
Form the dimension group $H = \lim A_n : \Z^{k_n} \to \Z^{k_{n+1}}$.  
\begin{enumerate}
\item  There is a unique positive group homomorphism $\Phi : H\to G$ such that $[e_i^{(n)},n]\mapsto a_i^{(n)}$; moreover, $\Phi(H^+) = G^+$.  
\item  If $\Phi$ is one to one then it is an isomorphism of ordered abelian groups.  
\end{enumerate}
\end{lemma}

\begin{proof}[Proof of Proposition \ref{PROP:characterization}, ``if'']
To prove the ``if'' direction, suppose that $F = \{ x_1,\ldots,x_k\} \subset G^+$ is a finite set satisfying the increasing monoid condition for an order automorphism $\alpha$ of $G$.  
We may suppose that $F$ does not contain $0$.  
By the ``increasing'' part of the increasing monoid condition, each element of $F$ can be written as a non-negative integer combination of elements of $\alpha(F)$.  
Let $A$ be a transition matrix representing these combinations; that is, $x_i = \sum_{j=1}^k (A)_{ji}\alpha(x_j)$ with $(A)_{ji}\in\Z^+$.  
The coefficients $(A)_{ji}$ are not unique if $k$ exceeds the rank of $G$, which is at most $k$ as $G = G^+-G^+$ implies that any finite independent subset is contained in $\langle \alpha^n(F)\rangle$ for some $n$.  

Left multiplication by the $k\times k$ matrix $A$ is a positive homomorphism $A : \Z^k\to\Z^k$ that sends the standard basis element $e_i$ to $\sum_{j=1}^k (A)_{ij}e_j$; let us also denote this homomorphism by $A$.  
Consider the stationary dimension group $H := \lim A : \Z^k\to\Z^k$.  
By Lemma \ref{LEM:increasing-semigroups} there is a positive homomorphism $\Phi : H\to G$ sending $[e_i,n] \mapsto \alpha^n(x_i)$, and if $\Phi$ is one to one then it is an order isomorphism.  
But $\Phi$ need not be one to one, so the remainder of the proof describes how to choose $A$ in such a way as to make $\Phi$ one to one.  
In particular, it will suffice to choose $A$ in such a way that the kernel of the homomorphism $\phi : e_i\mapsto x_i$ is also the kernel of $A$ ($\phi$ appears in Diagram \ref{DIAG:alpha-inverse}, below).  
This new choice of $A$ might have negative entries, so a further modification, using the unique trace property, will be required to ensure that its left and right eigenvectors associated to the weak Perron--Frobenius eigenvalue have strictly positive entries.  
Lemma \ref{LEM:eventually-positive} will then suffice to show that some power of $A$ is strictly positive, which is enough to prove the result.  

The transition matrix $A$ makes the following diagram commute.  

\begin{equation}
\xymatrix{
\Z^k \ar@{>}^{A}[r] \ar@{>}_{\phi}[d] & \Z^k \ar@{>}^{\phi}[d]  \\
G \ar@{>}_{\alpha^{-1}}[r] & G
}\label{DIAG:alpha-inverse}
\end{equation}

Any other transition matrix $A'$ for $\alpha$ differs from $A$ by an integer matrix, the columns of which lie in $\ker \phi$.  
$\ker\phi$ is a subgroup of $\Z^k$, and hence free abelian; moreover it is unperforated, meaning that $ng\in\ker\phi$ for $g\in\Z^k$ and $n\in\N$ implies that $g\in\ker\phi$.  
This implies that $\Z^k/\ker\phi$ is torsion-free, and hence is itself a free group, and so the exact sequence $0\to \ker\phi \to \Z^k\to \Z^k/\ker\phi \to 0$ splits, and $\Z^k\cong \ker\phi \oplus \Z^k/\ker\phi$.  
Then it is possible to extend a basis $\{ v_1,\ldots, v_l\}$ of $\ker\phi$ to a basis of $\Z^k$; let $L$ denote the integer matrix, the columns of which are the elements of this basis, starting with $\{ v_1,\ldots,v_l\}$.  
Because $L$ represents a basis for $\Z^k$, it is invertible over $\Z$.  

Let $D_l$ denote the $k\times k$ diagonal matrix, the first $l$ diagonal entries of which are $1$ and the last $k-l$ of which are $0$.  
Then the idempotent $LD_lL^{-1}$ leaves all elements of $\ker\phi$ fixed, and its column space is contained in $\ker\phi$.  
The commutativity of Diagram \ref{DIAG:alpha-inverse} implies that $A(\ker\phi)\subset \ker\phi$.  
So let $A' = A-LD_lL^{-1}A$; then the column space of $A-A'$ lies in $\ker\phi$, so $A'$ is again a transition matrix for $\alpha$, although possibly one with negative entries.  
Moreover, if $v\in\ker\phi$, then $A'v = Av - LD_lL^{-1}Av = Av-Av = 0$.  

The matrix $A'$ also acts linearly by left multiplication on the vector space $\R^k = \Z^k\otimes \R$; likewise $\alpha^{-1}$ induces an invertible linear operator (also denoted by $\alpha^{-1}$) on the vector space $G\otimes \R$, $\phi$ induces a linear map $\phi : \R^k\to G\otimes\R$, and the following diagram commutes.  

\begin{equation}
\xymatrix{
\R^k \ar@{>}^{A'}[r] \ar@{>}_{\phi}[d] & \R^k \ar@{>}^{\phi}[d]  \\
G\otimes\R \ar@{>}_{\alpha^{-1}}[r] & G\otimes\R
}\label{DIAG:real-tensor}
\end{equation}

Let $p$ denote the characteristic polynomial of the linear operator $\alpha^{-1}$ on $G\otimes \R$.  
$\alpha^{-1}$ is invertible, so the constant coefficient $\mu$ of $p$ is non-zero.  

Let $r$ denote the rank of $G$ and choose $r$ linearly independent elements of $G\otimes \R$.  
$p(\alpha^{-1})$ sends all of these elements to $0$.  
Pick one of these elements and express it as a real combination of $\{x_i\}_{i=1}^k$: $\sum_{i=1}^k c_ix_i$.  
Then $p(A')\sum_{i=1}^kc_ie_i = v$ for some $v\in\ker\phi$, so $p(A')$ $(\sum_{i=1}^kc_ie_i$ $- \frac{1}{\mu}v) = 0$.  
This yields $r$ linearly independent elements of $\R^k$ that go to $0$ under $p(A')$; along with the basis $\{ v_1,\ldots, v_l\}$ of $\ker\phi$ this gives a linearly independent set of size $r+l = k$.  
Since $\ker\phi$ is a $0$-eigenspace of $A'$, this means that the characteristic polynomial of $A'$ is $x^lp(x)$.  
Thus the eigenvalues of $A'$ are all the eigenvalues of $\alpha^{-1}$, with multiplicity, along with $0$, which has multiplicity $l = k-r$.  
The fact that $A$ is an integer matrix means that $p$ has integer coefficients, so in particular $|\mu|\geq 1$.  

The unique (up to multiplication by a positive real scalar) trace $\tau$ on $G$ can be extended to an element of $(G\otimes \R)^*$, the dual of the vector space $G\otimes\R$.  
Taking transposes in Diagram \ref{DIAG:real-tensor} yields the following commuting diagram.  

\begin{equation}
\xymatrix{
(\R^k)^*  & (\R^k)^* \ar@{>}^{A'^*}[l] \\
(G\otimes\R)^* \ar@{>}^{\phi^*}[u] & (G\otimes \R)^* \ar@{>}^{(\alpha^{-1})^*}[l] \ar@{>}_{\phi^*}[u] 
}\label{DIAG:dual}
\end{equation}

$\alpha^{-1}$ has a weak Perron--Frobenius eigenvalue $\lambda$ and $\tau$ is an eigenvector of $(\alpha^{-1})^*$ corresponding to that eigenvalue; this fact is proved below in Lemma \ref{LEM:PF}.  
Because $|\mu| = |\det (\alpha^{-1})| \geq 1$, it must be true that $\lambda > 1$.  
Thus $(\alpha^{-1})^*(\tau) = \lambda\tau$, and $\phi^*$ is injective (because $\phi$ is surjective), so $A'^*(\phi^*(\tau)) = \lambda \phi^*(\tau)$.  
Hence $\phi^*(\tau)$ is a $\lambda$-eigenvector of $A'^*$, which is the same as a left $\lambda$-eigenvector of $A'$.  

$\tau$ is a positive functional and each $x_i\in G^+$, so $\tau(x_i)\geq 0$, and because $G$ is simple, $\tau (x_i)$ is strictly positive, as mentioned in Section \ref{SEC:notation}.  
Then $\phi^*(\tau)$ is a row vector, the $i$th entry of which is $(\phi^*(\tau))(e_i) = \tau(x_i) > 0$.  

This row vector is a left $\lambda$-eigenvector of any integer matrix $A$ that makes Diagram \ref{DIAG:alpha-inverse} commute.  
Now it remains to show that $A'$ can be modified in such a way that it still satisfies $A'(\ker\phi) = \{ 0\}$ and it has a right $\lambda$-eigenvector with strictly positive entries.  

Let us replace $A'$ with $A'' = (A')^m + v_1w_1^t + \cdots + v_lw_l^t$, where each $w_i$ is an integer vector.  
The columns of the matrices $v_iw_i^t$ all lie in $\ker\phi$, so $A''$ again makes Diagram \ref{DIAG:alpha-inverse} commute (with $\alpha^{-1}$ replaced by $\alpha^{-m}$).  
Also, $\ker\phi$ has dimension $l$, so there are $r = k-l$ linearly independent integer row vectors $w$ satisfying $wv = 0$ for all $v\in\ker\phi$.  
Let $(\ker\phi)^\perp$ denote the set of all such row vectors and choose each $w_i^t$ from this set; this guarantees that $A''v = 0$ for all $v\in\ker\phi$.  

Let $v_0$ be a right $\lambda$-eigenvector for $A'$.  
Let us choose the rows $w_i^t$ in such a way that 
\begin{enumerate}
\item  $v_0 + s_1v_1+\cdots +s_lv_l$ is a $\lambda^m$-eigenvector of $A''$ for some real numbers $s_1,\ldots,s_l$; and 
\item  this vector has strictly positive entries.  
\end{enumerate}

Under the hypothesis that $w_i^t\in (\ker\phi)^\perp$, condition (1) becomes 
\begin{align*}
((A')^m+v_1w_1^t + \cdots + v_lw_l^t)\big( v_0+\sum_{j=1}^l s_jv_j\big)  & = \lambda^m\big( v_0+\sum_{j=1}^l s_jv_j\big) \\
\lambda^mv_0 + \sum_{j=1}^l (w_j^tv_0)v_j  & = \lambda^m\big( v_0+\sum_{j=1}^l s_jv_j\big) \\
\implies s_j & = \frac{w_j^tv_0}{\lambda^m}.
\end{align*}

$\phi(v_0)\in G\otimes\R$ is a right $\lambda$-eigenvector of $\alpha^{-1}$, and we may assume by replacing $v_0$ with $-v_0$ if necessary that $\tau(\phi(v_0)) > 0$.  
This means that $\phi(v_0)$ is strictly in the interior of $(G\otimes \R)^+ := \{ x\in G\otimes \R : \tau(x) > 0\} \cup \{ 0\}$ (which is the smallest real cone in $G\otimes \R$ containing $G^+$), so by the increasing monoid condition, $\phi (v_0)$ can be expressed as a positive real combination of $\alpha^n(x_1),\ldots, \alpha^n(x_k)$ for some sufficiently large $n$.  
Then, if we replace $F$ with $\{ \alpha^n(x_1),\ldots, \alpha^n(x_k)\}$, so that $\phi : e_i\to \alpha^n(x_i)$, this means that condition (2) is satisfied for at least one tuple $(s_1,\ldots,s_l)^t\in \R^l$.  

Let $S = \{ (s_1,\ldots,s_l)^t\in\R^l : v_0+s_1v_1+\cdots + s_lv_l $ is strictly positive $\}$.  
Then $S$ is convex: if 
\begin{align*}
v_0 + s_1v_1 + \cdots + s_lv_l & \quad \text{and} \\
v_0 + s_1'v_1 + \cdots + s_l'v_l & 
\end{align*}
are strictly positive and $0 \leq r \leq 1$, then 
\begin{align*}
& \phantom{=} v_0 + (rs_1+(1-r)s_1')v_1 + \cdots + (rs_l+(1-r)s_l')v_l \\
& = r(v_0+s_1v_1+\cdots +s_lv_l) + (1-r)(v_0+s_1'v_1+\cdots + s_l'v_l)
\end{align*}
is a sum of two strictly positive vectors, and hence is strictly positive.  

Assuming that $k>r$ (which we may as well do, otherwise the original choice of $A$ would have been sufficient), the set $\{x_1,\ldots,x_k\}$ is not linearly independent in $G\otimes \R \cong \R^r$, so any vector in the interior of the positive cone generated by these elements can be represented as a positive combination of them in at least $k-r=l$ linearly independent ways.  
(To see this, apply \cite[Exercise 2.36]{L:convex} to the convex hull of $\{ 0, \mu x_1,\ldots, \mu x_k\}$, where $\mu$ is large enough that this hull contains the given interior point.)  
Therefore $S$ is a convex subset of $\R^l$ with interior.  

There must be some $w^t\in(\ker\phi)^\perp$ such that $w^tv_0\neq 0$, for otherwise $v_0$ would be in $\ker\phi$.  
Then, as $(\ker\phi)^\perp$ is spanned by integer vectors, it is possible to choose an integer row vector $w_0^t\in (\ker\phi)^\perp$ with $w_0^tv_0\neq 0$.  
Because $\lambda > 1$, there exists $m \in \N$ such that $S$ contains an element of the lattice $\frac{w_0^tv_0}{\lambda^m}\Z^l$; call this element $\frac{w_0^tv_0}{\lambda^m}(a_1,\ldots,a_l)^t$ with $a_i\in\Z$.  
Then 
\[
A'' = (A')^m + v_1(a_1w_0)^t + \cdots + v_l(a_lw_0)^t
\]
has strictly positive left and right $\lambda$-eigenvectors, and hence by Lemma \ref{LEM:eventually-positive} some positive integer power of it has strictly positive entries, and hence represents a positive homomorphism $\Z^k\to\Z^k$.  
\end{proof}

The following lemma was used in the proof of the ``if'' part of Proposition \ref{PROP:characterization}.  

\begin{lemma}\label{LEM:PF}
Let $G$ be a simple dimension group with trace $\tau$ that is unique up to multiplication by a positive scalar.  
Suppose there exist an order automorphism $\alpha:G\to G$ and a finite subset of $G^+$ satisfying the increasing monoid condition.  
Then $\alpha^{-1}$ has weak a Perron--Frobenius eigenvalue $\lambda$, and, when viewed as an element of $(G\otimes\R)^*$, $\tau$ is a $\lambda$-eigenvector of $(\alpha^{-1})^* : (G\otimes \R)^*\to (G\otimes \R)^*$.  
\end{lemma}
\begin{proof}
It is clear from the hypotheses that $\tau$ is an eigenvector of $\alpha^*$ (and hence of $(\alpha^{-1})^*$) because $\tau\circ \alpha : G\otimes \R\to \R$ restricts to an order-preserving homomorphism of $G$ into $\R$, which is a positive multiple of $\tau$ by assumption.  

Now let us show that the eigenvalue $\lambda$ associated to $\tau$ is a weak Perron--Frobenius eigenvalue of $\alpha^{-1}$.  
Suppose for a contradiction that there is some other complex eigenvalue $\mu$ of $(\alpha^{-1})^*$ such that $|\mu|\geq \lambda$.  
Then there is some complex-valued linear homomorphism $\gamma : G\otimes\R\to \C$ such that $\gamma(\alpha^{-1}(g)) = \mu\gamma(g)$ for all $g\in G$.  

Let $F = \{ g_1,\ldots , g_k\}\subset G^+$ be a finite set such that $(\alpha,F)$ satisfies the increasing monoid condition.  
Then, as mentioned in Section \ref{SEC:notation}, the simplicity of $G$ implies that $\tau(g_i)>0$ for all $i\leq k$.  
Thus we may choose $M > 0$ such that $\tau(g_i) > M |\gamma(g_i)|$ for all $i\leq k$.  
Now if $h_1,h_2\in G^+$ satisfy $\tau(h_i)>M|\gamma(h_i)|$, then 
\begin{align*}
\tau(h_1+h_2)  = \tau(h_1)+\tau(h_2)  & \geq M|\gamma(h_1)| + M|\gamma(h_2)| \\
& \geq M|\gamma(h_1)+\gamma(h_2)| \\
& = M|\gamma(h_1+h_2)|.
\end{align*}
Therefore $\tau(g) > M |\gamma(g)|$ for all non-zero $g\in \Mon (F)$.  

Further, if $l\in\N$ and $g\in \alpha^l(\Mon (F))$, say $g = \alpha^l(g')$ with $g'\in\Mon (F)$, then 
\begin{align*}
\tau(g) = \tau(\alpha^l(g')) & = \lambda^{-l}\tau(g') \\
& \geq |\mu|^{-l} M |\gamma(g')| \\
& = M |\gamma(\alpha^{l}(g'))| = M |\gamma(g)|.
\end{align*}
Thus this same inequality holds for all elements of $\alpha^l(\Mon(F))$ with $l\in \N$.  

But this contradicts the hypothesis that $G^+ = \bigcup_l \alpha^l(\Mon(F))$, because every ball of sufficiently large radius in $G\otimes\R$ contains an element of $G$, and so the open half space $\{ x\in G\otimes \R : \tau(x)>0\}$ certainly contains an element of $\{ g\in G : M|\gamma(g)| > \tau(g)\}$, which is the intersection of $G$ with a union of half spaces in $G\otimes \R$.  
\end{proof}

\section{An intrinsic characterization of stationarity for unordered torsion-free abelian groups}\label{SEC:unordered}

The same arguments that were used in \cite{H:almost-ultrasimplicial} to prove Lemma \ref{LEM:increasing-semigroups} can also be used to prove Lemma \ref{LEM:increasing-subgroups}, below, which is the corresponding statement for inductive limits in the category of unordered torsion-free abelian groups.  

\begin{lemma}\label{LEM:increasing-subgroups}
Suppose that $G$ is a torsion-free abelian group with an increasing sequence of finitely-generated subgroups, $G_1\subset G_2\subset \cdots $ such that $G = \bigcup G_n$, with each $G_n$ generated by $\{a_i^{(n)}\}_{i=1}^{k_n}$.  
Suppose that $A_n$ is a transition matrix associated to this choice of generators for $G_n\subset G_{n+1}$; i.e., $a_i^{(n)} = \sum_{j=1}^{k_{n+1}}(A_n)_{ji}a_j^{(n+1)}$.  
Form the group $K = \lim A_n : \Z^{k_n} \to \Z^{k_{n+1}}$.  
\begin{enumerate}
\item  There is a unique group homomorphism $\Phi : K\to G$ such that $[e_i^{(n)},n]\mapsto a_i^{(n)}$; and 
\item  If $\Phi$ is one to one then it is an isomorphism of abelian groups.  
\end{enumerate}
\end{lemma}

There is likewise a notion of a stationary unordered torsion-free abelian group.  
\begin{defn}\label{DEF:stationary}
A torsion-free abelian group is \defemph{stationary} if it is isomorphic to the inductive limit of a stationary sequence in the category of torsion-free abelian groups.  
\end{defn}

Stationarity can be characterized intrinsically using the following condition.  
\begin{defn}\label{DEF:subgroup-condition}
Let $G$ be a torsion-free abelian group, let $\alpha: G\to G$ be an automorphism, and let $S\subset G$ be a subset.  
Let us say that the pair $(\alpha,S)$ satisfies the \defemph{increasing subgroup condition} if $G$ is the increasing union of the subgroups $\langle\alpha^n(S)\rangle$.  
\end{defn}

\begin{prop}\label{PROP:unordered}
Let $G$ be a torsion-free abelian group.  
Then $G$ is stationary if and only if it has an automorphism $\alpha$ and a finite subset $F\subset G$ that satisfy the increasing subgroup condition.  
\end{prop}
\begin{proof}
The proof uses the same arguments as the proof of Proposition \ref{PROP:characterization}, but for the ``if'' case, it suffices to stop once the matrix $A'$ has been constructed.  
\end{proof}

\begin{remark}\label{REM:subgroup-rank}
Note that it is always possible to choose a finite subset $F$ in Proposition \ref{PROP:unordered} that has a number of elements equal to the rank of $G$.  
This is because the subgroup $\langle \alpha^n(F)\rangle$ is the same as $\alpha^n(\langle F\rangle )$, and so we can replace $F$ with any basis for the free abelian group $\langle F\rangle$; such a basis necessarily has no more than $\rank G$ elements.  
It is also clear that such a basis can have no fewer than $\rank G$ elements, as $\rank G \leq \sup_n \rank \langle \alpha^n(F)\rangle = \rank \langle F\rangle$.  
\end{remark}

One direction of Theorem \ref{THM:description}, the main theorem, says that for a stationary simple dimension group the kernel of the trace is stationary in the category of unordered torsion-free abelian groups.  
The other direction of the main theorem says that the kernel of the trace can be any stationary unordered abelian group; Proposition \ref{PROP:unordered} will be useful in proving this statement.

\section{A concrete description of simple stationary limits}\label{SEC:description}

The main result of this section is Theorem \ref{THM:description}, which describes a simple stationary dimension group in terms of a simple stationary totally ordered dimension group (the image of the trace) and a stationary unordered group (the kernel of the trace).  
The following lemmas will be useful for this purpose.  

\begin{lemma}\label{LEM:quotient}
Let $A$ be a $k\times k$ primitive integer matrix with irrational Perron--Frobenius eigenvalue $\lambda$.  
Let $F\subset \Z^k$ be an $A$-invariant rank-$n$ subgroup such that $\Z^k/F$ is free and $F\otimes \R$ has trivial intersection with the $\lambda$-eigenspace of $A$.  
Let $A'$ denote a matrix representing the homomorphism induced by $A$ on $\Z^k/F$ with respect to some basis.  
Then there exists $P\in SL(n,\Z)$ such that $PA'P^{-1}$ is primitive.  
\end{lemma}
\begin{proof}
Because $F\otimes \R$ has trivial intersection with the $\lambda$-eigenspace of $A$, $\lambda$ is also an eigenvalue of $A'$.  
Since the characteristic polynomial of $A'$ divides that of $A$, $\lambda$ is also a weak Perron--Frobenius eigenvalue of $A'$.  
$A'$ has integer entries, so any right $\lambda$-eigenvector of $A'$ necessarily has at least two entries, the ratio of which is irrational.  
By \cite[Theorem 2.2]{H:irrational}, this implies the existence of $P$.  
\end{proof}
\begin{remark}\label{REM:rational-quotient}
The conclusion of Lemma \ref{LEM:quotient} is no longer true if the hypothesis that $\lambda\notin\Q$ is dropped.  
To see this, consider the primitive matrix 
\[
A = \left( \begin{array}{rrr}
1 & 2 & 2 \\
1 & 4 & 0 \\
1 & 0 & 4 
\end{array} \right)
\]
modulo the invariant subgroup $F = \langle (-4,1,1)^t\rangle$.  
With respect to the basis $(1,0,0)^t + F$ and $(0,1,0)^t+F$ of $\Z^3/F$, the induced homomorphism $A'$ has matrix 
\[
A' = \left( \begin{array}{rr}
5 & 2 \\
0 & 4
\end{array}\right),
\]
which is not similar to any primitive matrix, as it has weak left and right Perron--Frobenius eigenvectors $(1,2)$ and $(1,0)^t$ respectively, the product of which is $1 < 2$.  
From the discussion following Theorem 2.2 in \cite{H:irrational}, in order for an $n\times n$ integer matrix $A'$ with an integer weak Perron--Frobenius eigenvalue to be similar to a primitive matrix, its left and right Perron--Frobenius eigenvectors, when expressed in lowest terms, must have product with modulus at least $n$.  
(This condition was shown in \cite{H:rational} to be sufficient as well.)  
\end{remark}

\begin{lemma}\label{LEM:eigenvector-entries}
Let $A$ be a primitive $k\times k$ matrix with integer entries and Perron--Frobenius eigenvalue $\lambda$.  
Then a $\lambda$-eigenvector, right or left, of $A$ can be chosen such that its entries form a $\Q$-basis for $\Q(\lambda)$, the field extension of $\Q$ by $\lambda$.  
\end{lemma}
\begin{proof}
If $\lambda\in\Z$, this is immediate.  
If $\lambda\notin \Z$, then $\lambda$ is irrational, and the claim is a consequence of the classification of simple stationary totally ordered groups in \cite{H:irrational}; this can be seen in the following way.  

Form the dimension group $G = \varinjlim A:\Z^k\to \Z^k$, and consider the trace $\tau : G\to \R$ arising from a positive left $\lambda$-eigenvector $w$ of $A$, as discussed in Section \ref{SEC:notation}.  
Let us denote the kernel of $\tau$ by $w^\perp$.  
Then $\tau(G)\subset \R$ is a simple ordered group; let us show that it is also stationary in the category of ordered abelian groups.  

As discussed in \cite[Section 3]{H:irrational}, $\Z^k/w^\perp$ is torsion-free, say of rank $n$, so $A$ induces a homomorphism $A'$ on $\Z^k/w^\perp$; this homomorphism can also be represented by an integer matrix, albeit one that is not necessarily primitive.  
$\tau(G)$ is isomorphic as a group to the limit of $A':\Z^k/w^\perp \to \Z^k/w^\perp$ in the category of unordered abelian groups.  
The pre-image of $\R^+\setminus \{ 0\}$ under this isomorphism consists of all $[g+w^\perp,i]$ with $i\in\N$ and $g+w^\perp\in \Z^k/w^\perp$ for which $w'(g+w^\perp) > 0$, where $w'$ is the left Perron--Frobenius eigenvector of $A'$ that corresponds to $w$ (i.e., the pullback of $w$).  
Taking the positive cone to be this pre-image, together with $0$, makes $\varinjlim A':\Z^k/w^\perp \to \Z^k/w^\perp$ into an ordered group.  

By Lemma \ref{LEM:quotient}, there exists $P\in SL(n,\Z)$ such that $PA'P^{-1}$ is primitive.  
Then $\varinjlim PA'P^{-1}:\Z^n\to \Z^n$ is order isomorphic to $\varinjlim A':\Z^k/w^\perp \to \Z^k/w^\perp$ with the order just described, and hence also to $\tau(G)$.  
Thus $\tau(G)$ is stationary as an ordered abelian group.  

But then by \cite{H:irrational}, Theorem 3.3 and the discussion at the end of Section 4, there exists $r\in\R^+$ such that $r\tau(G)\otimes\Q\subset \R$ is a field; this field necessarily contains $\lambda$, which is the image of $1$ under the automorphism $\alpha$ of $r\tau(G)$ induced by $A'$, as in the proof of the ``only if'' part of Proposition \ref{PROP:characterization}.  
Moreover, the dimension of this field over $\Q$ equals the rank of $A'$, which is the algebraic degree of $\lambda$ because, by the discussion at the beginning of Section 3 of \cite{H:irrational}, the characteristic polynomial of $A'$ is irreducible.  
Thus $r\tau(G)\otimes \Q = \Q(\lambda)$.  

This means that $rw$ has all of its entries in $\Q(\lambda)$ because these entries are precisely the products of $w$ with the standard basis elements; these products necessarily lie in $r\tau(G)\subset \Q(\lambda)$.  
The fact that these entries span $\Q(\lambda)$ follows from the fact that they generate the same additive subgroup of $\R$ as the entries of $rw'$, left multiplication by which maps $\Z^k/w^\perp$ injectively into $\R$.  

To prove the same claim for a right $\lambda$-eigenvector of $A$, simply repeat this argument using the transpose of $A$.  
\end{proof}

\begin{remark}
The conclusion of Lemma \ref{LEM:eigenvector-entries} is not obvious (at least to me), and relies on results from \cite{H:irrational}, which in turn require that $A$ be primitive and have integer entries.  
To see that these hypotheses cannot be dropped entirely, note that $\left( \begin{array}{rr} 1 \\ \pi \end{array}\right)$ is a $2$-eigenvector of both %
$\left( \begin{array}{rr} %
2 & 0 \\
0 & 2 
\end{array}\right)$ %
and %
$\left( \begin{array}{cc} %
2 - \pi/2 & 1/2 \\
\pi & 1 
\end{array} \right)$.  

Still, using \cite[Theorem 2.2]{H:irrational}, the hypothesis in Lemma \ref{LEM:eigenvector-entries} that $A$ be primitive can be weakened to the requirement that $A$ have a weak Perron--Frobenius eigenvalue, although this complicates the proof somewhat and is not necessary in what follows.  
\end{remark}

The purpose of the following lemma is to identify a copy of $\tau(G)$ inside of $G$, where $\tau: G \to \R$ is the trace.  

\begin{lemma}\label{LEM:quotient-complement}
Let $A$ be a $k\times k$ primitive integer matrix with Perron--Frobenius eigenvalue $\lambda$ and corresponding left eigenvector $w$.  
Then there is an $A$-invariant subgroup $F\subset \Z^k$, of rank equal to the algebraic degree of $\lambda$ over $\Q$, such that $\Z^k/F$ is free abelian and the homomorphism $\Z^k\to \R$ given by left multiplication by $w$ is injective on $F$.  
\end{lemma}
\begin{proof}
Let $n$ denote the degree of $\lambda$ over $\Q$.  
By Lemma \ref{LEM:eigenvector-entries} $A$ has a right $\lambda$-eigenvector $v$, the entries of which span $\Q(\lambda)$ over $\Q$.  

Let $L$ denote the smallest normal field extension of $\Q$ containing $\lambda$; then $v\in L^k$ and $A$ can be viewed as a linear operator on $L^k$.  
Let $\alpha$ be a field automorphism of $L$ that fixes $\Q$, and for $x\in L^k$ let $\alpha(x)$ denote the vector in $L^k$, the entries of which are the images of the entries of $x$ under $\alpha$.  
The entries of $v$ are rational polynomials in $\lambda$, and the statement that $v$ is a $\lambda$-eigenvector of $A$ is equivalent to saying that $\lambda$ satisfies a system of $k$ rational polynomials.  
But if $\lambda$ satisfies these $k$ polynomials, then so does $\alpha(\lambda)$, and so $\alpha(v)$ is an $\alpha(\lambda)$-eigenvector of $A$.  

Let $V\subset L^k$ denote the subspace spanned by all $\alpha(v)$ as $\alpha$ ranges over all such field automorphisms.  
Then $V$ is spanned by the $n$ vectors obtained from $v$ by replacing $\lambda$ with each of its $n$ algebraic conjugates, and hence has dimension at most $n$.  
Let us show that $V$ contains $n$ linearly independent vectors with rational entries.  

Choose a basis $\zeta_1,\ldots,\zeta_n$ for $\Q(\lambda)$ over $\Q$ with $\zeta_1 = 1$.  
Let us perform a sequence of alterations to these elements without changing the property that they form a basis for $\Q(\lambda)$.  

Let $\Tr$ denote the field trace of $L$ over $\Q$, i.e., the $\Q$-linear map $\Tr : L \to \Q$ sending any element to the sum of all its images under embeddings of $L$ in $\C$, which coincide with automorphisms of $L$ because $L$ is normal \cite[Chapter 2]{M:number-fields}.  
Let $m$ denote $\Tr(1) = [L:\Q]$.  
Then for each $i > 1$, replace $\zeta_i$ with $\zeta_i - \frac{\Tr(\zeta_i/\zeta_1)}{m}\zeta_1$.  
These new elements $\zeta_1,\ldots,\zeta_n$ still form a basis of $\Q(\lambda)$ over $\Q$, and moreover now $\Tr(\zeta_i/\zeta_1) = 0$ for all $i > 1$.  

Now for each $i > 2$, replace $\zeta_i$ with $\zeta_i - \frac{\Tr(\zeta_i/\zeta_2)}{m}\zeta_2$.  
Proceeding in this fashion, we obtain a basis $\zeta_1,\ldots,\zeta_n$ of $\Q(\lambda)$ over $\Q$ with the desirable property that $\Tr(\zeta_i/\zeta_j) = $ if $i > j$.  

The statement that the entries of $v$ lie in $\Q(\lambda)$ means that there exist vectors $v_1,\ldots,v_n$ with rational entries such that $v = \zeta_1v_1 + \cdots + \zeta_nv_n$.  
The statement that the entries of $v$ span $\Q(\lambda)$ over $\Q$ means that $v_1,\ldots,v_n$ are linearly independent over $\Q$.  
For any automorphism $\alpha$ of $L$ fixing $\Q$, the vector $\alpha (\frac{1}{\zeta_j}v)$ is an $\alpha(\lambda)$-eigenvector of $A$, and hence lies in $V$.  
The sum of these vectors over all such automorphisms $\alpha$ for a fixed $j$ is 
\begin{equation}\label{EQ:field-trace}
\sum_\alpha \alpha\big(\frac{1}{\zeta_j}v\big) = \Tr(\zeta_1/\zeta_j)v_1 + \cdots + \Tr(\zeta_n/\zeta_j)v_n \in \spa_\Q\{v_1,\ldots,v_j\};
\end{equation}
moreover, the $v_j$-coefficient of this vector is $\Tr(\zeta_j/\zeta_j) = \Tr(1) = m$, which is non-zero.  
This is sufficient to show that the $n$ rational vectors $v_1,\ldots,v_n$ lie in $V$, and hence must span it (over $L$).  

Let $F$ = $V\cap \Z^k$.  
$F$ is necessarily $A$-invariant as $V$ and $\Z^k$ are $A$-invariant, and $\Z^k/F$ is free abelian by the definition of $F$.  
Moreover, $F$ has rank $n$ as it contains non-zero multiples of each of $v_1,\ldots, v_n$.  

To see that left multiplication by $w$ is injective on $F$, pick some element $y\in F$ and suppose $wy = 0$.  
Let $\alpha_1,\ldots,\alpha_n$ be automorphisms of $L$ fixing $\Q$ and sending $\lambda$ to each of its $n$ algebraic conjugates, and suppose $\alpha_1$ fixes $\lambda$.  
$y$ can be expressed uniquely as an $L$-linear combination of $\alpha_1(v),\ldots, \alpha_n(v)$, and for $i > 1$ $\alpha_i(v)$ is a right eigenvector of $A$ associated to $\alpha_i(\lambda)\neq \lambda$, so $w\alpha_i(v) = 0$ for $i > 1$.  

Thus the statement that $wy = 0$ is equivalent to saying that the coefficient of $v = \alpha_1(v)$ in the expansion of $y$ with respect to $\alpha_1(v),\ldots, \alpha_n(v)$ is $0$.  
But inspection of Equation \ref{EQ:field-trace} reveals this to be impossible.  
This is because we can write $y = c_1v_1+\cdots c_nv_n$ as a rational combination of the rational vectors $v_i$, and each $v_i$ can in turn be written as an $L$-linear combination of $\{ \alpha_1(v),\ldots,\alpha_n(v)\}$.  
The coefficients of $v=\alpha_1(v)$ in these combinations can in turn be expressed as rational combinations of the basis $\{ \frac{1}{\zeta_1},\ldots,\frac{1}{\zeta_n}\}$ of $L$ over $\Q$.  
Equation \ref{EQ:field-trace} implies that, for $v_i$, this coefficient of $\alpha_1(v)$ uses a non-zero multiple of $\frac{1}{\zeta_i}$ and zero multiples of $\frac{1}{\zeta_j}$ for $j>i$, which is sufficient to prove that the coefficient of $\alpha_1(v)$ in $y$ is non-zero.  
\end{proof}

Now let us suppose that $G$ is simple and stationary, and use this to find a description of $G$; this discussion will culminate in the statement and proof of Theorem \ref{THM:description}.  

Let $A$ be a $k\times k$ integer matrix, some positive integer power of which has strictly positive entries.  
Form the dimension group $G = \lim A : \Z^{k} \to \Z^{k}$.  
Let $\lambda$ denote the Perron--Frobenius eigenvalue of $A$, let $n$ denote its algebraic degree over $\Q$, and let $w$ and $v$ denote left and right Perron--Frobenius eigenvectors of it, respectively, chosen with positive entries.  
Let $\tau : G\to \R$ denote the trace on $G$ given by $w$, as described in Section \ref{SEC:notation}.  
Let us now describe the structure of $G$ in terms of subgroups that arise from $w$.  

If $\lambda = 1$, this structure is particularly easy to describe.  
Let $\mu$ be another eigenvalue of $A$; then $\mu$ satisfies the characteristic polynomial of $A$, and hence is an algebraic integer.  
Then the minimal polynomial of $\mu$ over $\Q$ is a monic integer polynomial dividing the characteristic polynomial of $A$.  
The constant coefficient of this minimal polynomial is the product of all the algebraic conjugates of $\mu$, which are also eigenvalues of $A$; assuming $1$ to be the Perron--Frobenius eigenvalue, these conjugates must all be less than $1$ in modulus, and so their product must also be less than $1$ in modulus.  
Since this constant coefficient is an integer, it must be $0$, which implies that $\mu= 0$.  
Then in this case it is not hard to see that $G$ is cyclic.  

Now suppose $\lambda  > 1$.  
Define a subgroup $L\subset \Z^k$ by $L = w^\perp := \{ x\in\Z^k : wx = 0\}$.  
Then $L$ is free abelian and $AL\subset L$.  
By Lemma \ref{LEM:eigenvector-entries}, the rank of $L$ is $k-n$.  

Define a subgroup $K\subset G$ by $K := \{ [g,i] : g\in L\}$.  
$K$ is indeed a subgroup because $AL\subset L$.  
Then $K$ is the inductive limit of the following stationary sequence in the category of (unordered) torsion-free abelian groups.  

\begin{equation*}
\xymatrix{
L \ar@{>}^{A|_L}[r] & L \ar@{>}^{A|_L}[r] & L \ar@{>}^{A|_L}[r] & \cdots
}
\end{equation*}

$K$ has finite rank---let us denote it by $r$---and is isomorphic to a subgroup of $\Q^r$.  
Let $\Phi : K \to \Q^r$ denote an embedding of $K$ as a subgroup of $\Q^r$.  

Of course $K$ is the kernel of the trace on $G$.  
The short exact sequence associated with the trace need not split, so it is not necessarily possible to realize the image of the trace as a summand complementing $K$; nevertheless, it is possible to find a copy of the image of the trace sitting inside $G$ as an order subgroup.  
To do this, let us use the $A$-invariant subgroup $F\subset \Z^k$ given by Lemma \ref{LEM:quotient-complement}.  

Define a subgroup $H\subset G$ by $H := \{ [g,i] : g\in F\}$; as a group this is the inductive limit in the category of torsion-free abelian groups of the stationary sequence $A|_F:F\to F$.  
By \cite[Theorem 2.2]{H:irrational} there exists a basis of $F$ with respect to which the matrix representing the group homomorphism $A|_F$ is primitive; this yields an order structure on $H$ making it a stationary simple dimension group.  
But this order structure agrees with the order that $H$ inherits as a subgroup of $G$ because both of these order structures are determined by traces induced by multiplication with left Perron--Frobenius eigenvectors---one an eigenvector of $A|_F$ and the other of $A$---and these are the same homomorphisms on $F$.  
Therefore $H$ is a simple stationary ordered group under the order structure that it inherits as a subgroup of $G$.  
Moreover left multiplication by $w$ is injective on $F$, so $\tau$ is injective on $H$ and $H$ is order isomorphic to $\tau(H)$.  

The fact that left multiplication by $w$ is injective on $F$ also means that $F\cap L = \{ 0\}$; combined with the fact that $F$ has rank $n$ and $L$ has rank $k-n$, this implies that $F+L$ is a sublattice of $\Z^k$ of full rank.  
This means that for every $g\in G$ there is an integer $m$ such that $mg\in H+K$.  
Thus the embeddings $\Phi$ and $\tau$ can be combined and extended by linearity to produce an embedding $\Theta : G \to \R\oplus \Q^r$ defined as follows: If $g\in G$ satisfies $mg = h + k$ for $h\in H, k\in K$, then $\Theta(g) = (\frac{1}{m}\tau(h),\frac{1}{m}\Phi(k))$.  

Let us now show that $G$ is generated by $H$, $K$, and a finite number of extra elements.  
To do this, choose elements $z_1,\ldots, z_s\in\Z^k$ such that $\Z^k = \langle z_1,\ldots, z_s,L,F\rangle$; then $[z_1,1],\ldots,[z_s,1]$ are the required extra generators of $G$.  
To see this, consider the subgroup $L_i$ of $\Z^k$ defined by $L_i := \langle A^i(\Z^k) \cup L \cup F \rangle$.  
$L_i$ is a sublattice of $\Z^k$ of full rank.  
$L_1\supset L_2\supset L_3 \supset \cdots$, and this sequence must eventually stabilize, as each $L_i$ contains $\langle L\cup F\rangle$, so the modulus of the determinant of $L_i$ is bounded above by the modulus of the determinant of that lattice.  
Thus there exists $N\in\N$ such that $L_N = L_{N+i}$ for all $i\in \N$.  

Then pick $[g,i]\in G$.  
Note that $A^Ng\in L_N = L_{N+i-1}$, so $A^Ng = A^{N+i-1}z + x + y$ for some $z\in \Z^k, x\in L$, and $y\in F$.  
Thus
\begin{align*}
[g,i] & = [A^Ng,N+i] \\
& = [A^{N+i-1}z + x + y, N+i] \\
& = [A^{N+i-1}z,N+i] + [x,N+i] + [y,N+i] \\
& = [z,1] + [x,N+i] + [y,N+i].
\end{align*}

The second of these summands lies in $K$ and the third lies in $H$.  
The first of these summands lies in the first copy of $\Z^k$ in the inductive sequence, and hence is a combination of elements from $\{ [t,1] : t\in L\} \subset K$, $\{ [u,1] : u\in F\}\subset H$, and $\{ [z_1,1], \ldots,[z_s,1]\}$.  
Therefore $G$ is generated by $H, K$, and $\{ [z_1,1], \ldots, [z_s,1]\}$.  

It is clear that the number $s$ of extra generators can be taken to be at most $k$, but in fact we can do better.  
$\Z^k/L$ is free, so the exact sequence $L\to \Z^k \to \Z^k/L$ splits, and any basis of $L$ can be extended to a basis of $\Z^k$; the $n$ extra elements used to extend this basis can be taken to be the extra generators.  
Thus $s$ can be taken to be $n = \rank (F) = \rank (H)$.  

Likewise $\Z^k/F$ is free, so we could take $\{ z_1,\ldots, z_s\}$ to be a set of elements used to extend a basis of $F$ to a basis of $\Z^k$.  
Therefore $s$ can be taken to be at most $k-n = \rank(L)$.  
But we can do still better.  
By passing to a sufficiently high power of $A$, we can guarantee that $\ker(A|_L^2) = \ker(A|_L)$, and then $r = \rank(K) = \rank(A|_L)$.  
Then $L/\ker (A|_L)$ is free, so any basis of $\langle F\cup \ker(A|_L)\rangle$ can be extended to a basis of $\Z^k$; the extra elements used to extend this basis can be used as $\{ [z_1,1],\ldots, [z_s,1]\}$, and we need use at most $r = \rank(K)$ such extra elements.  
Therefore the number $s$ of extra elements used to generate $G$ can be taken to be at most the minimum of $\rank(H)$ and $\rank(K)$.  

As described in Section \ref{SEC:notation}, every trace on $G$ is a positive multiple of the homomorphism $[g,i] \mapsto \lambda^{-i}wg$.  
The pullback of such a trace under the embedding $\Theta : G \to \R\oplus \Q^{r}$ is a positive scalar multiple of the projection map on the first coordinate, because $wx = 0$ for all $x\in L$.  
Moreover, an element $[g,i]\in G$ is positive if and only if $wg > 0$ or $[g,i] = 0\in G$, so if $\R\oplus\Q^{r}$ is ordered by the first coordinate---i.e., $q = (x,q_1,\ldots,q_r) \in \R\oplus\Q^{r}$ is positive if and only if $x>0$ or $q = 0$---then $\Theta$ is an order embedding.  

This discussion proves one direction of the following theorem, which is the main result.  

\begin{thm}\label{THM:description}
Let $G$ be a non-cyclic simple dimension group.  
Then $G$ is stationary if and only if it is order isomorphic to a subgroup of $\R \oplus \Q^{r}$ ordered by the first coordinate and generated by the following:  
\begin{enumerate}
\item  a non-cyclic stationary order subgroup $H\subset \R\oplus 0^r$;  
\item  a rank-$r$ subgroup $K \subset \{ (0,q_1,\ldots,q_r) \in\R\oplus \Q^{r} \}$ that is stationary in the category of unordered torsion-free abelian groups; and
\item  a finite set $\{ z_1,\ldots, z_s\}\subset (H + K)\otimes \Q\subset \R \oplus \Q^r$.  
\end{enumerate}
Moreover, the number $s$ of extra generators can be taken to be less than or equal to the minimum of the ranks of $H$ and $K$.  
\end{thm}

\begin{defn}\label{DEF:brace}
In the setting of Theorem \ref{THM:description}, let us refer to the elements $z_1,\ldots, z_s$ as \defemph{braces}.  
\end{defn}

The proof of the ``if'' part of Theorem \ref{THM:description} uses Corollary \ref{COR:matrix-power-difference}, which is a consequence of Proposition \ref{PROP:polynomials}, below.  
The proof of Proposition \ref{PROP:polynomials} uses the following lemma, which involves a brief excursion into the world of symmetric polynomials.  

\begin{lemma}\label{LEM:symmetric-polynomials}
Let $f(x) = (x-z_1)\cdots (x-z_n)$ be a polynomial with complex roots $z_1,\ldots, z_n$, not necessarily distinct, and for $l\in\Z^+$ let $f^l(x)$ denote the polynomial $(x-z_1^l)\cdots (x-z_n^l)$.  
Let $e_{j,l}$ denote the coefficient of $x^{n-j}$ in $f^l(x)$ (so that $e_{j,1}$ is the coefficient of $x^{n-j}$ in $f(x)$).  
Suppose that each $e_{j,1}\in\Z$.  
Then
\begin{enumerate}
\item  each $e_{j,l}\in\Z$; and
\item  if $p$ is an integer prime such that $p | e_{j,1}$ for all $j > j_0$, then, for all $l>1$, $p|e_{j_0,l}$ if and only if $p|e_{j_0,1}$.  
\end{enumerate}
\end{lemma}
\begin{proof}
View the complex roots $z_i$ as variables, and consider the ring of symmetric polynomials in $z_1,\ldots, z_n$, that is, the subring $\Lambda \subset \Z[z_1,\ldots,$ $z_n]$ consisting of all integer polynomials that are invariant under permutations of $\{z_i\}_{i=1}^n$.  
The coefficients of $f(x)$ are precisely the \defemph{elementary symmetric polynomials} 
\begin{equation*}
e_{j,1} = \sum_{1\leq i_1<i_2<\cdots <i_j\leq n}z_{i_1}z_{i_2}\cdots z_{i_j}.
\end{equation*}

The Fundamental Theorem of Symmetric Polynomials \cite[Theorem 2.4]{M:symmetric-functions} says that $\Lambda = \Z[e_{0,1},\ldots,e_{n,1}]$, and the set $\{e_{j,1}\}_{j=0}^n$ is algebraically independent over $\Z$.  
The coefficients $e_{j,l}$ of $f^l(x)$ lie in $\Lambda$ because $f^l(x)$ itself is invariant under permutation of $\{z_i\}_{i=1}^n$; therefore these coefficients can be expressed as integer polynomials in $\{e_{j,1}\}_{j=0}^n$.  
Thus if each $e_{j,1}$ is an integer, then so is each $e_{j,l}$, proving statement (1).  

To prove statement (2), note that $e_{j_0,l}-e_{j_0,1}^l$ is symmetric, and so can be written as an integer polynomial in $\{e_{j,1}\}_{j=0}^n$.  
Consider the evaluation map $\phi$ on $\Lambda = \Z[e_{0,1},\ldots,e_{n,1}]$ obtained by setting $z_j = 0$ for all $j > j_0$.  
This is a unital ring homomorphism $\phi : \Z[e_{0,1},\ldots,e_{n,1}] \to \Z[\epsilon_{0,1},\ldots,\epsilon_{j_0,1}]$, where $\epsilon_{j,1}$ is the $j$th elementary symmetric polynomial in the variables $\{z_i\}_{i=1}^{j_0}$.  
$\phi(e_{j,1}) = \epsilon_{j,1}$ for $j\leq j_0$, and $\phi(e_{j,1}) = 0$ for $j>j_0$.  
As the elementary symmetric polynomials are algebraically independent over $\Z$, the kernel of $\phi$ consists exactly of the $\Z$-submodule of $\Z[e_{0,1},\ldots,e_{n,1}]$ spanned by the monomials that are divisible by at least one $e_{j,1}$ with $j>j_0$.  
The observation that $\phi(e_{j_0,l}-e_{j_0,1}^l) = $ $z_1^lz_2^l\cdots z_{j_0}^l - (z_1z_2\cdots z_{j_0})^l = 0$ is then sufficient to prove statement (2).  
\end{proof}

\begin{prop}\label{PROP:polynomials}
Let $f(x)\in\Z[x]$ be a monic polynomial and let $m \geq 2$ be an integer.  
Then there exist $g(x), r(x) \in\Z[x]$ and integers $k > l \geq 0$ such that
\[
f(x)g(x) + mr(x) = x^k - x^l.
\]
\end{prop}
\begin{proof}
Let $f(x) = x^n + a_1x^{n-1} + \cdots + a_{n-1}x+a_n$, and let us first prove the claim for the case when $(a_n,m) = 1$.  

Rearrange the formula for $f(x)$ to obtain
\begin{equation}\label{EQ:units}
xs(x) + f(x) = a_n
\end{equation}
for some $s(x)\in\Z[x]$.  

Consider the two natural unital quotient homomorphisms $q_1: \Z[x]\to (\Z/m\Z)[x]$ and $q_2: (\Z/m\Z)[x]\to R = (\Z/m\Z)[x]/\langle q_1(f(x))\rangle$, and for $h(x)\in\Z[x]$ let $\overline{h(x)}$ denote $q_2(q_1(h(x))$.  
The assumption that $(a_n, m) = 1$ implies that $a_n + m\Z$ is a unit in the finite ring $\Z/m\Z$, and hence $q_1(a_n)\in (\Z/m\Z)[x]$ is also a unit, and hence so is $\overline{a_n}$.  
Then Equation \ref{EQ:units} implies that $\overline{x}\overline{s(x)} = \overline{a_n} \in R^\times$, and so $\overline{x}\in R^\times$.  

But $R$ is a finite ring because $f(x)$ is monic: any polynomial in $(\Z/m\Z)[x]$ is equivalent modulo $q_1(f(x))$ to a polynomial of degree less than $n$, and the coefficient ring $(\Z/m\Z)$ is finite.  
Thus $R^\times$ is a finite group, and hence there exists $k\geq 0$ such that $\overline{x}^k = \overline{1}$.  
Lifting to a pre-image of $\overline{x}^k-\overline{1}$ in $\Z[x]$ yields the desired result with $l = 0$.  

Now drop the assumption that $(a_n,m)=1$, and let $j_0$ be maximal with the property that $(a_{j_0},a_{j_0+1},\ldots,a_n,m) = 1$.  
We know that $j_0\geq 0$ because $a_0 = 1$.  
Moreover, the argument above shows that the claim is true for degree-$n$ polynomials $f(x)$ for which $j_0 = n$.  

Let us show that, given a monic degree-$n$ polynomial $f_i(x)$ with $j_0 = j_0(i) \geq i$, there exist $g_i(x)\in\Z[x]$, $c_i,l_i\in\Z$ with $l_i>0$ such that
\begin{align}\label{EQ:induction}
f_i(x)g_i(x) + mc_i & = x^{l_i}f_{i+1}(x^{l_i}),
\end{align}
where $f_{i+1}(x)$ is a monic degree-$n$ polynomial with $j_0(i+1) = j_0(i) + 1 \geq i+1$.  
Then the constant coefficient of $f_n(x)$ is coprime to $m$, and so by the argument above $f_n(x)$ satisfies the conclusion of the proposition; induction then suffices to show that the conclusion holds for $f(x) = f_0(x)$.  

Note that $f_i(x)|f_i^l(x^l)$ for all $l>0$, where $f_i^l(x)$ is defined as in Lemma \ref{LEM:symmetric-polynomials} (and is an integer polynomial by part (1) of that lemma).  
Let $e_{n,1} = f_i(0)$ and choose $l_i$ large enough that $e_{n,1}^{l_i}$ ``consumes'' all of the powers of the common prime factors of $e_{n,1}$ and $m$; that is, with $m' = m/(e_{n,1}^{l_i},m)$, we have $(e_{n,1}^{l_i},m') = 1$.  

Then $f_i^{l_i}(0) = e_{n,1}^{l_i}$, and the constant coefficient of $(x-m')f_i^{l_i}(x)$ is $-m'e_{n,1}^{l_i}$, which is divisible by $m$.  
Let $c_i = m'e_{n,1}^{l_i}/m\in\Z$, and let 
\begin{align}\label{EQ:next-f}
f_{i+1}(x) & = \frac{(x-m')f_i^{l_i}(x)+m'e_{n,1}^{l_i}}{x}\in\Z[x];
\end{align}
then the statement in Equation \ref{EQ:induction} holds with this choice of $c_i$ and $f_{i+1}(x)$ (and some choice of $g_i(x)\in\Z[x]$).  
It remains to show that $j_0(i+1) = j_0(i) + 1$ (which is at least $i+1$ by induction).  

Let $f_i(x) = x^n+e_{1,1}x^{n-1}+\cdots +e_{n-1,1}x+e_{n,1}$ and let $f_i^{l_i}(x) = x^n+e_{1,l_i}x^{n-1}+\cdots +e_{n-1,l_i}x+e_{n,l_i}$.  
By inspecting Equation \ref{EQ:next-f}, one sees that the coefficient of $x^{n-j}$ in $f_{i+1}(x)$ is $b_j=-m'e_{j-1,l_i}+e_{j,l_i}$.  
By the definition of $j_0(i)$, there is a prime $p$ dividing $m$ and each $e_{j,1}$ for $j>j_0(i)$; by statement (2) of Lemma \ref{LEM:symmetric-polynomials}, $p$ also divides $e_{j,l_i}$ for $j>j_0(i)$.  
Therefore $p|-m'e_{j,l_i}+e_{j+1,l_i}$ for all $j>j_0(i)$; i.e., $p|b_{j+1}$ for all $j>j_0(i)$, which means that $j_0(i+1)\leq j_0(i)+1$.  

Now suppose that $p$ is a prime dividing $(b_{j_0(i)+2},\ldots,b_n,m)$, and let us show that such a prime necessarily divides each $e_{j,l_i}$ with $j>j_0(i)$.  
Indeed, $l_i$ was chosen with the property that any prime $p$ dividing $m$ divides either $m'$ or $e_{n,l_i}$, but not both.  
If $p|m'$, then $p|b_n = -m'e_{n-1,l_i}+e_{n,l_i}$ implies $p|e_{n,l_i}$, which is a contradiction; therefore $p\ndiv m'$ and $p|e_{n,l_i}$.  
But then by induction $p|e_{j,l_i}$ for all $j>j_0(i)$, where the induction step uses the argument that $p|e_{j+1,l_i}$, $p\ndiv m'$, and $p|b_{j+1} = -m'e_{j,l_i}+e_{j+1,l_i}$ together imply that $p|e_{j,l_i}$.  

Finally, using statement (2) of Lemma \ref{LEM:symmetric-polynomials} and the definition of $j_0(i)$, we see that $p\ndiv e_{j_0(i),l_i}$; combined with the statements that $p\ndiv m'$ and $p|e_{j_0(i)+1,l_i}$, this means that $p\ndiv b_{j_0(i)+1}$.  
Thus $j_0(i+1)=j_0(i)+1$.  
\end{proof}

\begin{cor}\label{COR:matrix-power-difference}
Let $B$ be a square integer matrix and let $m\geq 2$ be an integer.  
Then there exist integers $k > l\geq 0$ such that every entry of the matrix $B^k - B^l$ is divisible by $m$.  
\end{cor}
\begin{proof}
Let the characteristic polynomial of $B$ play the role of $f$ in the statement of Proposition \ref{PROP:polynomials} and apply the Cayley--Hamilton Theorem.  
\end{proof}

\begin{cor}\label{COR:fun-with-graphs}
Let $G$ be a finite directed graph and let $m\geq 2$ be an integer.  
Then there exist integers $k>l\geq 0$ such that, for any two vertices $i$ and $j$ of $G$, the number of paths of length $k$ from $i$ to $j$ differs from the number of paths of length $l$ from $i$ to $j$ by a multiple of $m$.  
\end{cor}

Now let us give a proof of the ``if'' part of Theorem \ref{THM:description}.  
This proof is broken into a sequence of lemmas for readability.  

Let $Z$ denote the set of braces $\{ z_1,\ldots,z_s\}$, and throughout the proof let us use the symbol $z$ to denote an element of $Z$.  

Let us assume that $G$ has the form described in Theorem \ref{THM:description}, and let us produce an order automorphism $\gamma:G\to G$ and a finite set $F\subset G^+$ that satisfy the increasing monoid condition; Proposition \ref{PROP:characterization} will then suffice to show that $G$ is stationary.  
$\gamma$ and $F$ will be constructed from an automorphism $\alpha$ and finite subset $F_1$ of $H$ satisfying the increasing monoid condition---we know by Proposition \ref{PROP:characterization} that these exist---and an automorphism $\beta$ and finite subset $F_2$ of $K$ that satisfy the increasing subgroup condition---we know by Proposition \ref{PROP:unordered} that these exist.  

The order automorphism $\gamma$ will have the following form on the subgroup $H+K$: $\gamma(h+k) = \alpha^{l_1}(h)+\beta^{l_2}(k)$ for $h\in H, k\in K$, with $l_1,l_2\in\N$.  
Such a formula has a unique extension by linearity to an order automorphism of $\R\oplus \Q^{r}$ (ordered by the first coordinate).  
Note that this order automorphism restricts to automorphisms of $H$ and $K$; to show that it restricts to an order automorphism of $G$, it is necessary and sufficient to show that, with the proper choice of $l_1$ and $l_2$, the elements $\gamma(z),\gamma^{-1}(z)$ lie in $G$, where $z$ ranges over the set $Z$ of braces.  
Lemma \ref{LEM:automorphism-extends}, below, is a slightly stronger statement, which we will anyway need later.  

\begin{lemma}\label{LEM:automorphism-extends}
Let $z\in (H + K)\otimes \Q$.  
There exist $c_1,c_2\in\N$ such that, for any $a,b\in\N$, the extension to $(H+K)\otimes\Q$ of the group homomorphism $\gamma$ defined on $H+K$ by $\gamma(h+k) = \alpha^{ac_1}(h) + \beta^{bc_2}(k)$ satisfies
\[
\gamma(z) - z, \gamma^{-1}(z) - z \in H+K.
\]
\end{lemma}

\begin{proof}
There exist $m\in\Z^+$, $h\in H$, and $k\in K$ such that $z = \frac{1}{m}(h+k)$.  
Then
\begin{align}\label{EQ:gamma-z}
\gamma(z) - z & = \frac{1}{m} (\alpha^{c_1}(h)-h) + \frac{1}{m}(\beta^{c_2}(k)-k).  
\end{align}
Let us show that, with an appropriate choice of $c_2$, the second summand is an element of $K$.  
Because $(\beta,F_2)$ satisfy the increasing subgroup condition for $K$, $k\in \langle \beta^{c_0}(F_2)\rangle$ for some $c_0\in\N$.  
Let $k_1,\ldots, k_r$ denote the elements of $\beta^{c_0}(F_2)$, and let $B$ be a transition matrix for $\beta$, so that
\begin{align}\label{EQ:transition-matrix}
k_i & = \sum_{j=1}^r (B)_{ji}\beta(k_j).  
\end{align}

By Corollary \ref{COR:matrix-power-difference}, above, there exist integers $c_2'>c_2''\geq 0$ such that $B^{c_2''}-B^{c_2'}$ has every entry divisible by $m$.  
Let $c_2 = c_2'-c_2''$ and let $C\in \Z^r$ be a column of integers representing $k$ as a combination of $\{ k_1,\ldots, k_r\}$.  
$B^{c_2'}C$ represents $k$ as a combination of $\{ \beta^{c_2'}(k_1),\ldots,\beta^{c_2'}(k_r)\}$, and $B^{c_2''}C$ represents $\beta^{c_2}(k)$ as a combination of the same elements.  
Thus $\beta^{c_2}(k)-k\in m\langle \beta^{c_2'}(\{ k_1,\ldots, k_r\})\rangle$, and hence $\frac{1}{m}(\beta^{c_2}(k)-k)\in K$.  
Indeed, this is true if $c_2$ is any positive integer multiple of $c_2'-c_2''$.  

We can repeat the same argument to find $c_1$ that works for with $h$ and $\alpha$.  
Then a very similar argument, with the same choices of $c_1$ and $c_2$, works for $\gamma^{-1}(z)-z$.  
\end{proof}

Let $l_2$ denote the least common multiple of the $c_2$ values obtained in Lemma \ref{LEM:automorphism-extends} as $z$ ranges over $Z$, and let $l_1$ denote the least common multiple of the $c_1$ values.  
Then $\gamma : (H+K)\otimes \Q \to (H+K)\otimes \Q$ defined by $\gamma(\frac{1}{m} (h+k)) = \frac{1}{m} (\alpha^{l_1}(h) + \beta^{l_2}(k))$ restricts to an order automorphism of $G$.  
We could replace $l_1$ with any multiple of $l_1$ and $l_2$ with any multiple of $l_2$ and this would still be true.  

As $(\beta,F_2)$ satisfies the increasing subgroup condition, it is clear that $(\beta^{l_2},F_2)$ does as well; likewise $(\alpha^{l_1},F_1)$ satisfies the increasing monoid condition.  
Moreover, by replacing $F_2$ with $\beta^{il_2}(F_2)$ if necessary, we may suppose without loss of generality that, if $z = \frac{1}{m}(h+k)\in Z$, then $k\in \langle F_2\rangle$; likewise we may suppose that $h\in \langle F_1\rangle$.  
But even more, we may suppose without loss of generality that $h\in\Mon(F_1)$, i.e., $z\in G^+$.  
Indeed, if $z\in (-G^+)$, then we may replace $z$ with $-z$, while if $z\notin (-G^+)\cup G^+$, then $mz\in K$, and we may remove $z$ from the list $Z$ of braces and replace $K$ with the stationary group $\langle K\cup \{ z\} \rangle$ in the statement of the theorem.  

Now let us construct the set $F$.  
Let us fix notation for the elements of $F_1$ and $F_2$: $F_1 = \{ h_1,\ldots,h_n\}$ and $F_2 = \{ k_1,\ldots, k_r\}$.  
Choose $h_0\in F_1$ and define positive elements $g_i := h_0 + k_i$ for $1\leq i\leq r$, and $g_{r+1} := h_0-k_1-\cdots - k_r$.  
Then define $F$ by 
\begin{equation}
F := F_1\cup Z \cup \{ g_1,\ldots, g_r\}.
\end{equation}

The goal is to show that $(\gamma,F)$ satisfies the increasing monoid condition, where $\gamma$ is constructed as above, possibly after replacing $l_1$ and $l_2$ with multiples.  
As a first step let us show that, with the right choices of $l_1$ and $l_2$, $(\gamma|_{H+K},F\setminus Z)$ satisfies the increasing monoid condition for the order subgroup $H+K\subset G$.  

\begin{lemma}\label{LEM:increasing-monoid-h-plus-k}
If $l_1$ and $l_2$ are chosen appropriately, then $(\gamma|_{H+K},F\setminus Z)$ satisfies the increasing monoid condition for the order subgroup $H+K\subset G$.  
Specifically, given $g\in (H+K)^+$, there exists $p_0>0$ such that, if $p\geq p_0$, then $g\in \Mon(\gamma|_{H+K}^p(F\setminus Z))$.  
\end{lemma}
\begin{proof}
The automorphism $\alpha$ of $H$ is given by division by some $\lambda > 1$, so that, for $h\in H$, $\gamma(h) = \frac{1}{\lambda^{l_1}}h$.  

Pick non-zero $g\in (H+K)^+$; then $g$ has the form 
\begin{align*}
g & = h + \beta^{q'}(c_1k_1+\cdots + c_rk_r) 
\end{align*}
with $q'\geq 0$, $c_i\in\Z$, and $h\in H^+$.  This expression is not unique: pick $q$ with $ql_2>q'$; then we can also write
\begin{align*}
g & = \alpha^{ql_1}(\lambda^{ql_1}h) + \beta^{ql_2}(\beta^{q'-ql_2}(c_1k_1+\cdots + c_rk_r)),
\end{align*}
where 
\begin{align*}
\beta^{q'-ql_2}(c_1k_1+\cdots + c_rk_r) & = d_1k_1+\cdots + d_rk_r
\end{align*}
with $d_i\in\Z$ as $q'-ql_2<0$.  
In fact, the coefficients $d_i$ are determined by the transition matrix $B$ of $\beta$ from Equation \ref{EQ:transition-matrix} by the following formula.  
\begin{align}\label{EQ:coeff-bound}
d_i & = \sum_{i=1}^r (B^{ql_2-q'})_{ij}c_j.
\end{align}
Let $M_c = \max_i\{ |c_i|\}, M_d = \max_i \{ |d_i|\}$, and $L = \max_{i,j}\{|(B)_{ij}|\}$.  
Then by Equation \ref{EQ:coeff-bound} $M_d\leq (rL)^{ql_2-q'}M_c$.  
Note that $M_d$ varies with $q$, but $M_c$ does not.  

Now write $d_1k_1+\cdots + d_rk_r$ as an integer combination of $F\setminus Z$ as follows.  
\begin{align*}
d_1k_1+\cdots d_rk_r & = (M_d+d_1)g_1 + \cdots + (M_d+d_r)g_r + M_dg_{r+1} \\
& \qquad -((r+1)M_d + \sum_{i=1}^r d_i)h_0.
\end{align*}
Note that all of these coefficients are non-negative, except perhaps for that of $h_0$.  
Then 
\begin{align*}
g & = \alpha^{ql_1}(\lambda^{ql_1}h) + \beta^{ql_2}(\beta^{q'-ql_2}(c_1k_1+\cdots + c_rk_r)) \\
& = \gamma^{q}(\lambda^{ql_1}h+d_1k_1+\cdots + d_rk_r) \\
& = \gamma^{q}\big( (M_d+d_1)g_1 + \cdots + (M_d+d_r)g_r + M_dg_{r+1} \\
& \qquad + \lambda^{ql_1}h-((r+1)M_d + \sum_{i=1}^r d_i)h_0\big).
\end{align*}

This is $\gamma^q$ of an integer combination of elements of $F\setminus Z$ plus the element $h' := \lambda^{ql_1}h - ((r+1)M_d+\sum_{i=1}^r d_i)h_0\in H\subset \R$.  
This is a totally ordered group, so we can determine conditions for $h'$ to be positive.  
\begin{align}
h' & = \lambda^{ql_1}h-((r+1)M_d + \sum_{i=1}^r d_i)h_0 \notag \\
 & \geq \lambda^{ql_1}h  - (2r+1)M_dh_0 \notag \\
\label{EQ:g-coeff}& \geq \lambda^{ql_1}h  - (2r+1)(rL)^{ql_2-q'}M_ch_0.
\end{align}

If we choose $l_1$ and $l_2$ in such a way that
\begin{align}
\frac{\lambda^{l_1}}{(rL)^{l_2}} & > 1 \notag \\
\label{EQ:l-ratio} l_1 & > l_2\frac{\log rL}{\log \lambda},
\end{align}
then $h'$ will be positive for all sufficiently large $q$.  
Indeed, we can say more precisely what is meant by ``sufficiently large $q$'' by declaring the quantity on the right hand side of Inequality \ref{EQ:g-coeff} to be positive and then solving for $q$:  
\begin{align}
\lambda^{ql_1}h  - (2r+1)(rL)^{ql_2-q'}M_ch_0 & > 0 \notag \\
\lambda^{ql_1}h  & > (2r+1)(rL)^{ql_2-q'}M_ch_0 \notag \\
\Big( \frac{\lambda^{l_1}}{(rL)^{l_2}}\Big)^{q} & > \frac{(2r+1)M_ch_0}{(rL)^{q'}h} \notag \\
\label{EQ:sufficient-q}q & > \frac{\log(2r+1) + \log \frac{M_ch_0}{(rL)^{q'}h}}{l_1\log \lambda - l_2\log rL}.
\end{align}

There are two facts that should be emphasized about Inequality \ref{EQ:sufficient-q}.  
Firstly, the quantities $\lambda,r,L,l_1$, and $l_2$ are all constants, and only the quantity $\log \frac{M_ch_0}{(rL)^{q'}h}$ on the right hand side depends on the group element $g$.  

Secondly, the direction of the inequality is preserved in the last line of Inequality \ref{EQ:sufficient-q} precisely because $l_1$ and $l_2$ satisfy Inequality \ref{EQ:l-ratio}.  
Moreover, such a choice of $l_1$ and $l_2$ can be made consistent with the requirements for $\gamma$ to be an order automorphism, because by Lemma \ref{LEM:automorphism-extends} those requirements allow $l_1$ to be chosen arbitrarily large independently of $l_2$.  

Let us use Inequality \ref{EQ:sufficient-q} to find a power $q$ that suffices for all $g_i$ with $1\leq i \leq r+1$.  
For these elements, $M_c = 1$, $h = h_0$, and $q' = 0$.  
Thus if $Q > \frac{\log (2r+1)}{l_1\log \lambda - l_2 \log rL}$ then for all $1\leq i \leq r+1$ we can write $g_i = \overline{h}_i + \overline{g}_i$, with $\overline{h}_i\in H^+$ and $\overline{g}_i \in \Mon(\gamma^q(\{ g_1,\ldots, g_{r+1}\}))$ for all $q\geq Q$.  
Use the increasing monoid condition for $(\alpha^{l_1},F_1)$ with respect to $H$ to choose $P\in \N$ such that for all $1\leq i \leq r+1$, $\overline{h}_i \in \Mon(\alpha^{pl_1}(F_1))$ for all $p\geq P$.  
Then for all $1\leq i\leq r+1$, $g_i\in \Mon(\gamma^q(F\setminus Z))$ for all $q \geq \max\{ P,Q\}$.  

Returning to the generic element $g\in G^+$, we know that there exists $q\in \N$ such that we can write $g = \gamma^q(e_1g_1+\cdots + e_{r+1}g_{r+1} + h')$ with $e_i\in\Z^+$ and $h'\in H^+$.  
Let $Q$ be as above and choose $T\in\N$ such that $h'\in \Mon(\alpha^{tl_1}(F_1))$ for all $t\geq T$; then $g\in \Mon(\gamma^p(F\setminus Z))$ if $p\geq q+\max\{ Q,T\}$.  

Technically this does not show that $(\gamma|_{H+K},F\setminus Z)$ satisfies the ``increasing'' part of the increasing monoid condition; however, it does if we replace $\gamma$ with $\gamma^{\max \{ P,Q\}}$, which can be done by modifying $l_1$ and $l_2$.  
\end{proof}

Next let us show that any brace $z$ lies in $\Mon(\gamma^p(F))$ for all sufficiently large $p$.  
\begin{lemma}\label{LEM:braces-in-monoid}
If $z\in Z$ is a brace, then there exists $P\in\N$ such that $z\in\Mon(\gamma^p(F))$ for all $p\geq P$.  
\end{lemma}

\begin{proof}
Let $z = \frac{1}{m}(h+k)$ with $h\in H^+$ and $k\in K$, as in Lemma \ref{LEM:automorphism-extends}.  
We can write an expression for $z$ similar to Equation \ref{EQ:gamma-z} from that lemma:  
\begin{align}
\label{EQ:z-gamma-mu}z & = \gamma^q (z) + \frac{1}{m}(h-\gamma^q(h)) + \frac{1}{m} (k-\gamma^q (k)) \\
& = \gamma^q (z) + g_q \notag
\end{align}
with $g_q\in (H+K)^+$ because $\gamma^q(h)<h$.  
Let us show that there is a power $P$ such that $g_q$ is in $\Mon(\gamma^p(F\setminus Z ))$ for all $q$ and for all $p \geq P$.  
If such a $P$ exists, then $z\in \Mon (\gamma^p(F))$ for all $p > P$.  

Such a $P$ must satisfy Inequality \ref{EQ:sufficient-q} from the proof of Lemma \ref{LEM:increasing-monoid-h-plus-k} regardless of the value of the right hand side, which depends upon $g_q$, and hence $q$.  
But only the term $\log \frac{M_ch_0}{(rL)^{q'}h}$ depends upon $g_q$; if we can find an upper bound for this term, say $M$, then the minimal integer $P$ satisfying $P > \frac{\log (2r+1) + M}{l_1\log \lambda - l_2\log rL}$ will suffice.  

So let us find an upper bound for this term.  
The symbols used in it do not have the same meanings in Lemma \ref{LEM:increasing-monoid-h-plus-k} as they do here.  
Here $\frac{1-\lambda^{-q l_1}}{m}h$ plays the roles of $h$, $q' = q l_2$, and $mM_c$ is the maximum modulus of any coefficient $c_i$ of an element $k_i$, $1\leq i\leq r$, in the expansion $\beta^{-q l_2}(k) - k = c_1k_1+\cdots c_rk_r$.  
Let $M_k$ denote the maximum modulus of any coefficient $d_i$ in the expansion $k = d_1k_1+\cdots + d_rk_r$.  
Then, by arguments similar to those used in the proof of Lemma \ref{LEM:increasing-monoid-h-plus-k}, $mM_c \leq ((rL)^{q l_2}+1)M_k$.  
Thus
\begin{align*}
\log \frac{M_ch_0}{(rL)^{q l_2}\frac{1-\lambda^{-q l_1}}{m}h} & \leq \log \frac{\frac{(rL)^{q l_2}+1}{m}M_k h_0}{(rL)^{q l_2}\frac{1-\lambda^{-q l_1}}{m}h} \\
& = \log \frac{(rL)^{q l_2}+1}{(rL)^{q l_2}} \frac{1}{1-\lambda^{-q l_1}} \frac{M_kh_0}{h} \\
& \leq \log 2 \frac{1}{1-\lambda^{-1}}\frac{M_kh_0}{h}.
\end{align*}

$M_k$ and $h$ are defined in terms of $z$, and so do not depend on $q$, and neither do $\lambda$ and $h_0$, which are fixed.  
Therefore this bound is independent of $q$, and hence so is the resulting power $P$.  
Therefore $z\in\Mon(\gamma^p (F))$ for all $p \geq P$.  
\end{proof}

It follows easily from Lemma \ref{LEM:braces-in-monoid} that $\Mon(\gamma^i(F))$ is an increasing sequence of monoids, perhaps after passing to a power of $\gamma$.  
Now all that remains is to show that the union of this sequence is all of $G^+$.  
\begin{lemma}\label{LEM:monoid-union}
The union of the increasing sequence $\Mon(\gamma^i (F))$ of monoids is all of $G^+$.  
\end{lemma}

\begin{proof}
Pick a non-zero $g\in G^+$; then $g$ can be written as a sum of elements of the subgroups $\langle z_1,\ldots,z_s\rangle,H$, and $K$.  
Thus
\begin{align*}
g & = a_1z_1 + \cdots + a_sz_s + h + k,
\end{align*}
with $h\in H$, $k\in K$, and $a_i\in\Z$ for all $i\leq s$.  

For $i\leq s$, we can find $m_i\in\N$, $h_i'\in H^+$, and $k_i'\in K$ such that $z_i = \frac{1}{m_i}(h_i'+k_i')$.  
Then the condition that $g\in G^+$ is equivalent to $\frac{a_1}{m_1}h_1'+\cdots + \frac{a_s}{m_s}h_s' + h > 0$ (as a real number).  
We may moreover suppose without loss of generality that each $a_i\geq 0$, because otherwise we may choose $b_i\in \N$ such that $a_i + b_im_i \geq 0$, and then $a_iz_i = (a_i+b_im_i)z_i -b_im_iz_i$, the latter summand of which lies in $H+K$, allowing us to replace $a_iz_i$ with $(a_i+b_im_i)z_i$ by modifying $h$ and $k$.  

Rewriting the expression for $g$ using Equation \ref{EQ:z-gamma-mu} yields, for arbitrary $q > 0$,  
\begin{align*}
g & = \sum_{i=1}^s \big( \gamma^q(a_iz_i) + \frac{a_i}{m_i}(h_i'-\gamma^q(h_i')) + \frac{a_i}{m_i}(k_i'-\gamma^q(k_i'))\big) + h + k,
\end{align*}
which, after rearranging, gives us
\begin{align*}
g - \sum_{i=1}^s \gamma^q (a_iz_i) & = h + \sum_{i=1}^s \frac{a_i}{m_i}(h_i'-\lambda^{-ql_1}h_i') + k + \sum_{i=1}^s \frac{a_i}{m_i}(k-\gamma^q(k_i')).  
\end{align*}
The quantity on the right hand side of this equation lies in $H+K$; moreover the summand from $H$ can be made arbitrarily close to $h+\frac{a_1}{m_1}h_1'+\cdots + \frac{a_s}{m_s}h_s'$, and hence will be positive for some sufficiently large $q$.  

Thus $g$ is a sum of elements of the form $\gamma^{q}(a_iz_i)$ and an element $g'\in (H+K)^+$.  
Lemmas \ref{LEM:increasing-monoid-h-plus-k} and \ref{LEM:braces-in-monoid} then suffice to show that $g\in \Mon (\gamma^p (F))$ for sufficiently large $p$.  
\end{proof}

This completes the proof of Theorem \ref{THM:description}.  
Let us make an observation about the result.  
\begin{remark}\label{REM:number-of-braces}
The proof of the ``if'' part of Theorem \ref{THM:description} will work for any finite number $s$ of braces; thus any group generated by $H$, $K$, and an arbitrary finite subset of $(H+K)\otimes \Q$ is stationary.  
The proof of the ``only if'' part then shows that, in fact, the number of braces need not exceed the minimum of the ranks of $H$ and $K$.  
\end{remark}

\section{Ordered cohomology groups of substitution tiling spaces}\label{SEC:tiling-cohomology}

Where are these ordered groups used?  
As mentioned in Section \ref{SEC:intro}, simple stationary dimension groups arise as invariants of subshifts of finite type.  
Let us show in this section that they also arise as the ordered top-level cohomology groups of substitution tiling spaces.  

The necessary background and definitions from the theory of tiling spaces can be found in \cite{AP}; let us concentrate here on the algebra.  
The tiling space arising from a substitution is a metric space that turns out to be homeomorphic to the inverse limit of a sequence of finite CW-complexes, called approximants, with maps between them, and in fact the CW-complexes and the maps can be taken to be the same at every stage.  
The \u{C}ech cohomology with integer coefficients of the tiling space is then isomorphic to the inductive limit of the \u{C}ech cohomology of the approximant under the homomorphism induced on cohomology by the self-map.  
The top-level cochain group has one generator for each prototile of the substitution (assuming that it forces its border---see \cite{AP}); let us denote the number of prototiles by $n$.  

An order structure can be defined on the top-level cohomology group by saying that a non-zero group element is positive if it is the pre-image of a strictly positive real number under the Ruelle--Sullivan map \cite{KP:ruelle-sullivan}.  
This definition also applies to other tiling spaces that do not arise from substitutions.  
Using \cite[Theorem 2.2]{EHS}, one can see that the ordered top-level cohomology group is a dimension group as long as the image of the Ruelle--Sullivan map is dense in $\R$, which, I suppose, must always be the case.  
However, it is not obvious that it is a stationary simple dimension group.  

This order structure has an alternative definition for substitution tiling spaces; this is the definition that was originally given in \cite{ORS:ordered-cohomology}.  
Under the assumption that the substitution is primitive---i.e., the associated $n\times n$ transition matrix $A$ has some power with strictly positive entries---the order structure is defined by declaring a non-zero element $[g,i]$ of the inductive limit group to be positive if $w'g > 0$, where $w'$ is a positive weak Perron--Frobenius left eigenvector of the integer matrix $A'$ representing the homomorphism induced on cohomology by the self-map.  
This sounds like the simple stationary dimension groups discussed in this paper and in \cite{H:irrational}, but there is one minor difference: the top-level cohomology group of the approximant is a quotient of a cochain group $C$ (on which the transition matrix $A$ acts) by a coboundary group that is invariant under $A$.  
This means that the matrix $A'$ represents the homomorphism induced by $A$ on a quotient of $C$, and so is not necessarily primitive, as $A$ is.  
Furthermore, although $C$ is free abelian, the quotient by coboundaries need not be---in particular, it might have torsion \cite{GHK:torsion}.  

Nevertheless, the torsion-free part of the ordered top-level cohomology group of the tiling space is, in many cases, and perhaps all cases, isomorphic to a simple stationary dimension group.  
Specifically, by Lemma \ref{LEM:quotient}, the matrix $A'$ is similar to a primitive matrix if its weak Perron--Frobenius eigenvalue is irrational, so in such cases the ordered top-level cohomology group is stationary and simple.  
But Remark \ref{REM:rational-quotient} gives an example that shows that the matrix induced by a primitive integer matrix on a free abelian quotient group is not necessarily primitive if the Perron--Frobenius eigenvalue is rational.  
The idea behind this example uses a result from \cite{H:rational} that says that an $n\times n$ integer matrix $A'$ is similar via an element of $SL(n,\Z)$ to a primitive matrix if and only if its left and right Perron--Frobenius eigenvectors $w$ and $v$, normalized so that both are unimodular (i.e., the greatest common divisor of the entries is $1$), satisfy $|wv|>n$.  
The matrix $A$ from Remark \ref{REM:rational-quotient} was chosen to have unimodular right Perron--Frobenius eigenvector $v = 5x_1+x_3$, where $x_1,x_3$ are part of a basis for $\Z^3$, so that the image of $v$ in the quotient $\Z^3/\langle x_3\rangle$ is no longer unimodular.  
But it might not be possible for this to happen if the quotient is taken modulo a group generated by coboundaries, so perhaps the ordered top-level cohomology group of a substitution tiling space is stationary and simple, although I do not know how to prove this.  

\begin{remark}\label{REM:range-of-tiling-invariant}
There are some questions that follow naturally from this discussion.  
The more difficult direction of Theorem \ref{THM:description} says, essentially, that any dimension group satisfying certain obvious necessary conditions can be realized as a stationary limit of some primitive integer matrix $A$ with an integer Perron--Frobenius eigenvalue.  
An arbitrary primitive integer matrix can be realized as the transition matrix of a primitive one-dimensional substitution, and by passing to a sufficiently high power of $A$, which does not change the resulting dimension group, we can even choose a substitution that forces its border and that contains every possible two-tile sequence in the resulting tilings.  
For such a substitution the approximant will be a wedge of circles, the subgroup of coboundaries will be trivial, and the matrix $A'$ induced on top-level cohomology will equal $A$.  
Thus any simple stationary dimension group can be realized as the ordered top-level cohomology group of some substitution tiling space.  
\end{remark}

There is also the question of torsion subgroups, which must necessarily be finitely generated.  
\begin{question}
Which torsion groups appear as summands of the top-level cohomology groups of tiling spaces?  
Does the presence of a given torsion subgroup impose any restrictions on the dimension group part?  
\end{question}

\section{Worked example}\label{SEC:example}

Let $B = \left( \begin{array}{cc} 1 & 1 \\ 1 & 4 \end{array} \right)$ and let $K$ be the stationary unordered abelian group  $\bigcup_{n=0}^\infty B^{-n}(\Z^2)\subset \Q^2$.  
Then let $H = \Z [1/5]$, which is clearly stationary as an ordered abelian group, and let $G \subset (H\oplus K)\otimes \Q \subset \Q^3$ be the group generated by $H\oplus 0$, $0\oplus K$, and $z = (1/2,1/2,1/2)$ and ordered by the first coordinate.  
Let us find a primitive integer matrix that realizes the stationary property for $G$.  

Note that, following Corollary \ref{COR:matrix-power-difference}, $B^3-I$ has all entries divisible by $2$, so an automorphism of $G$ should have the form $\alpha(h,k) = (h/5^n,B^{-3}k)$ (where $B^{-3} = \frac{1}{27} \left( \begin{array}{rr} 73 & -22 \\ -22 & 7 \end{array}\right)$).  
The power $n$ needs to be chosen large enough---certainly $5^n$ must exceed the maximum eigenvalue of $B^3$---and in this case $n = 3$ will suffice.  

Let us find a finite subset of $G^+$ that satisfies the increasing monoid condition with some power of $\alpha$.  
The following elements will suffice.  
\begin{align*}
x_1 & = (1, 0, 0) & \alpha(x_1) & = (1/125,0,0) \\
x_2 & = (1, 1, 0) & \alpha(x_2) & = (1/125,73/27,-22/27) \\
x_3 & = (1,-1, 0) & \alpha(x_3) & = (1/125,-73/27,22/27) \\
x_4 & = (1, 0, 1) & \alpha(x_4) & = (1/125,-22/27, 7/27) \\
x_5 & = (1, 0,-1) & \alpha(x_5) & = (1/125, 22/27,-7/27) \\
x_6 & = (1/2, 1/2,1/2) & \alpha(x_6) & = (1/250, 51/54,-15/54).
\end{align*}

By inspection we can find a preliminary transition matrix $A$ that expresses $x_i$ as an integer combination of $\alpha(x_j)$, as in the proof of Proposition \ref{PROP:characterization}.  
$\ker \phi$, also from the proof of that proposition, is spanned by the three elements given below.  
\[
A = \left[ \begin{array}{rrrrrr} %
125 & 96 & 96 & 30 & 30 & 1 \\
 0 &  7 &  0 & 22 &  0 & 14 \\
 0 &  0 &  7 &  0 & 22 &  0 \\
 0 & 22 &  0 & 73 &  0 & 47 \\
 0 &  0 & 22 &  0 & 73 &  0 \\
 0 &  0 &  0 &  0 &  0 &  1 %
\end{array}\right], 
\ker \phi   = \left\langle \left[ \begin{array}{r} 2 \\ -1 \\ -1 \\ 0 \\ 0 \\ 0 \end{array}\right],
 \left[ \begin{array}{r} 2 \\  0 \\  0 \\ -1 \\ -1 \\ 0 \end{array}\right],
 \left[ \begin{array}{r} -3 \\ 0 \\ 1 \\ 0 \\ 1 \\ 2 \end{array}\right] \right\rangle.
\]

$\ker\phi$ is not contained in $\ker A$, so let us replace $A$ with a different transition matrix that does satisfy this condition.  
This set of generators for $\ker\phi$ can be extended to a basis of $\Z^6$ by adding the elements $(1,0,0,0,0,0)^t$, $(0,0,0,1,0,0)^t$, and $(0,0,0,0,0,1)^t$; let $L$ denote the $6\times 6$ matrix, the columns of which are the entries of this basis.  
Let $D_3$ denote the diagonal matrix, the first three diagonal entries of which are $1$s and the last three of which are $0$s, as in the proof of Proposition \ref{PROP:characterization}.  
Then 
\[
A' := A-LD_3L^{-1}A = %
 \left[ \begin{array}{rrrrrr} %
125 & 103 & 147 & 52 & 198 & 15 \\
 0 &  0 &  0 &  0 &  0 &  0 \\
 0 &  0 &  0 &  0 &  0 &  0 \\
 0 & 15 &-15 & 51 &-51 & 33 \\
 0 &  0 &  0 &  0 &  0 &  0 \\
 0 & 14 &-14 & 44 &-44 & 29 %
\end{array}\right], 
\]
which has $\ker\phi$ in its kernel, but is certainly not primitive, and has weak right Perron--Frobenius eigenvector $(1,0,0,0,0,0)^t$.  

The rows of $A'$ generate the subspace $(\ker\phi)^\perp$, although only one of these rows has non-zero product with $(1,0,0,0,0,0)^t$, and that product is $125$, which is somewhat large.  
We can combine the rows and divide by $5$ to obtain the element $(25, 32, 18, 47, 3, 27)\in (\ker\phi)^\perp$.  

Let us now replace $A'$ with a primitive matrix $A''$ of the form
\[
(A')^m + %
\left( %
 a_1\left[ \begin{array}{r} 2 \\ -1 \\ -1 \\ 0 \\ 0 \\ 0 \end{array}\right] + 
 a_2\left[ \begin{array}{r} 2 \\  0 \\  0 \\ -1 \\ -1 \\ 0 \end{array}\right] + 
 a_3\left[ \begin{array}{r} -3 \\ 0 \\ 1 \\ 0 \\ 1 \\ 2 \end{array}\right] %
\right)(25,32,18,47,3,27).
\]

The coefficients $a_i\in\Z$ should have the property that, with $s_i = 25a_i/125^m$, 
\[
    \left[ \begin{array}{r} 1 \\  0 \\  0 \\ 0 \\ 0 \\ 0 \end{array}\right] + 
 s_1\left[ \begin{array}{r} 2 \\ -1 \\ -1 \\ 0 \\ 0 \\ 0 \end{array}\right] + 
 s_2\left[ \begin{array}{r} 2 \\  0 \\  0 \\ -1 \\ -1 \\ 0 \end{array}\right] + 
 s_3\left[ \begin{array}{r} -3 \\ 0 \\ 1 \\ 0 \\ 1 \\ 2 \end{array}\right] %
\]
has strictly positive entries.  
(Here the factor of $25$ in the numerator of $s_i$ comes from the fact that $(25,32,18,47,3,27)(1,0,0,0,0,0)^t = 25$.)  

Inspection reveals that $m$ must be at least $2$.  
Then there is a lot of leeway in how the $a_i$s are chosen: $a_1=a_2 = -60$ and $a_3 = 60$ will work.  
This results in the following matrix for $A''$.  
\[
A'' = %
 \left[ \begin{array}{rrrrrr} %
5125 &  425 & 9825 & -9928 & 20178 & -3071 \\
1500 & 1920 & 1080 &  2820 &   180 &  1620 \\
3000 & 3840 & 2160 &  5640 &   360 &  3240 \\
1500 & 3147 & -147 &  6873 & -3873 &  4260 \\
3000 & 3840 & 2160 &  5640 &   360 &  3240 \\
3000 & 4906 & 1094 &  9160 & -3160 &  5533 %
\end{array}\right]. 
\]

This is indeed primitive (the third power has strictly positive entries).  
Right eigenvectors of this matrix are $(41,12,24,12,24,24)^t$, $(2,-1,-1,0,0,0)^t$, $(2,0,0,-1,-1,0)$, $(-3,0,1,0,1,2)^t$, and 
\[
\left[ \begin{array}{r} -11227\pm 3071\sqrt{13}\\ 1600\mp 440\sqrt{13}\\ 3200\mp 880\sqrt{13}\\ 1609\mp431\sqrt{13}\\ 3200\mp 880\sqrt{13}\\ 3236\mp 880\sqrt{13}\end{array}\right].  
\]

These last two vectors span the same subspace as the integer vectors $y_1 = (-10220,1460,2920,1451,2920,2938)^t$ and $y_2 = (3071,-440,$ $-880,$ $-431,-880,-880)^t$.  
Let $\mathcal{B}$ denote the set consisting of these first four eigenvectors and these last two vectors; then $\mathcal{B}$ spans a sublattice of $\Z^6$ of determinant $2\cdot 3^4\cdot 5^3$.  
Let $C$ denote the integer matrix, the columns of which are the elements of $\mathcal{B}$.  
Then the entries of $C^{-1}$ are the coefficients of the standard basis vectors of $\Z^6$ when expressed as rational combinations of $\mathcal{B}$.  
With the exception of $(0,0,0,0,0,1)^t$, all of these standard basis elements can be represented as combinations in which the coefficient of $(41,12,24,12,24,24)^t$, the $15625$-eigenvector of $A''$, has coefficient $1/125$, and the vectors $y_1$ and $y_2$ have coefficients that are integer multiples of $1/9$.  
The coefficients of the $0$-eigenvectors of $A''$ do not matter.  

In the expansion of $(0,0,0,0,0,1)^t$, the $15625$-eigenvector has a coefficient of $1/250$, and $y_1$ and $y_2$ both have a coefficient of $1/18$.  
Moreover, $A''$ acts as $B^6$ on the subspace generated by $y_1$ and $y_2$; that is $A''y_1 = 533y_1+1760y_2$ and $A''y_2 = 1760y_1+5813y_2$.  
This is sufficient to show that $\varinjlim A'':\Z^6\to \Z^6$ is order isomorphic to $G$.  

\bibliographystyle{abbrv}
\bibliography{stationary}

\end{document}